\documentclass[12pt, reqno]{amsart}
\usepackage[utf8]{inputenc}
\usepackage{amsmath}
\usepackage{color}
\usepackage{amsfonts}
\usepackage{amssymb}
\usepackage{amsthm}
\usepackage{hyperref}
\newcommand{\eps}{\varepsilon}

\newtheorem{theorem}{Theorem}[section]
\newtheorem{proposition}[theorem]{Proposition}
\newtheorem{lemma}[theorem]{Lemma}

\newtheorem{remark}[theorem]{Remark}
\newtheorem{corollary}[theorem]{Corollary}
\theoremstyle{definition}

\numberwithin{equation}{section}

\title{Exceptional characters and nonvanishing of Dirichlet $L$-functions}
\author{Hung M. Bui}
\address{Department of Mathematics, University of Manchester, Manchester M13 9PL, UK}
\email{hung.bui@manchester.ac.uk}
\author{Kyle Pratt}
\address{ {Department of Mathematics \\
 University of Illinois at Urbana-Champaign \\
 1409 West Green Street, Urbana, IL 61801 \\ USA} }
\curraddr{ {All Souls College, Oxford OX1 4AL, United Kingdom} }
\email{{\href{mailto:kyle.pratt@all-souls.ox.ac.uk}{kyle.pratt@all-souls.ox.ac.uk}}, {\href{mailto:kvpratt@gmail.com}{kvpratt@gmail.com}}}
\author{Alexandru Zaharescu}
\address{Department of Mathematics, University of Illinois at Urbana-Champaign, 1409 West Green Street, Urbana, IL 61801, USA
and Simion Stoilow Institute of Mathematics of the Romanian Academy, P.O. Box 1-764, RO-014700 Bucharest,
Romania}
\email{zaharesc@illinois.edu}

\allowdisplaybreaks


\begin{document}

\maketitle

\begin{abstract}
Let $\psi$ be a real primitive character modulo $D$. If the $L$-function $L(s,\psi)$ has a real zero close to $s=1$, known as a Landau-Siegel zero, then we say the character $\psi$ is exceptional. Under the hypothesis that such exceptional characters exist, we prove that at least fifty percent of the central values $L(1/2,\chi)$ of the Dirichlet $L$-functions $L(s,\chi)$ are nonzero, where $\chi$ ranges over primitive characters modulo $q$ and $q$ is a large prime of size $D^{O(1)}$. Under the same hypothesis we also show that, for almost all $\chi$, the function $L(s,\chi)$ has at most a simple zero at $s = 1/2$.
\end{abstract}

\section{Introduction}

A central question in analytic number theory is to study the vanishing or nonvanishing of $L$-functions at special points. In some instances the vanishing or nonvanishing of an $L$-function has dramatic arithmetic implications, as described for instance by the Birch and Swinnerton-Dyer Conjecture [\textbf{\ref{Wiles}}].

We study the family of primitive Dirichlet $L$-functions modulo $q$ at the central point, where throughout this paper $q$ denotes a large prime. One conjectures that all of the central values $L(1/2,\chi)$ are nonzero; this would follow, for instance, from strong conjectures of Katz and Sarnak [\textbf{\ref{KS99}}] about symmetry types of families of $L$-functions. We are far from proving this conjecture. Even under the Generalized Riemann Hypothesis it is only known that [\textbf{\ref{Murty90}}]
\begin{align*}
\frac{1}{\phi(q)}\sideset{}{^*}\sum_{\substack{\chi(\text{mod }q) \\ L(1/2,\chi) \neq 0}}  1 \geq \frac{1}{2}-o(1),
\end{align*}
where $\sum^*$ denotes summation over all primitive characters. Khan and Ngo [\textbf{\ref{KN}}] showed unconditionally that
\begin{align*}
\frac{1}{\phi(q)}\sideset{}{^*}\sum_{\substack{\chi(\text{mod }q) \\ L(1/2,\chi) \neq 0}}  1 \geq \frac{3}{8}-o(1),
\end{align*}
using the mollification method and deep estimates for sums of Kloosterman sums. The proportion is slightly smaller for general large $q$ [\textbf{\ref{IS}}, \textbf{\ref{B}}].

In addition to studying the central values $L(1/2,\chi)$, one can also ask about the vanishing or nonvanishing of the central derivatives $L^{(k)}(1/2,\chi)$. It was shown by Bui and Milinovich [\textbf{\ref{BM}}] that
\begin{align}\label{eq:BM k deriv lower bound}
\frac{1}{\phi(q)}\sideset{}{^*}\sum_{\substack{\chi(\text{mod }q) \\ L^{(k)}(1/2,\chi) \neq 0}}  1 \geq 1-O(k^{-2})
\end{align}
as $k$ tends to infinity. Again it is believed that, for any nonnegative integer $k$, we have $L^{(k)}(1/2,\chi) \neq 0$ for almost all primitive $\chi$ modulo $q$. Michel and VanderKam [\textbf{\ref{MV}}] also proved earlier a result on the nonvanishing of high derivatives of the completed $L$-functions $\Lambda(s,\chi)$.

We are interested in connecting these questions of nonvanishing with certain hypothetical counterexamples to the Generalized Riemann Hypothesis. It is well-known that, with at most one real exception, there are no zeros of $L(s,\chi)$ in the region
\begin{align*}
\text{Re}(s) \geq 1 - \frac{c}{\log(q|\text{Im}(s)| + 2)}.
\end{align*}
(See [\textbf{\ref{Dav}}; Chapter 14].) If such a real zero exists then $\chi = \psi$ must be a real character. These real zeros would be among the strongest possible contradictions to the Generalized Riemann Hypothesis, and therefore there is great interest in showing these zeros do not exist. The existence or non-existence of these zeros is intimately related to the size of $L(1,\psi)$ and the class number $h(D)$. The lower bound
\begin{align}\label{eq:class number lower bound}
L(1,\psi) \gg D^{-1/2},
\end{align}
which one obtains from Dirichlet's class number formula, is not strong enough for most applications. Landau [\textbf{\ref{L}}]
improved this lower bound, showing
\begin{align*}
L(1,\psi) \gg_\varepsilon D^{-3/8-\varepsilon},
\end{align*}
and Siegel [\textbf{\ref{Si}}] showed that
\begin{align*}
L(1,\psi) \gg_\varepsilon D^{-\varepsilon}.
\end{align*}
The lower bounds of Landau and Siegel are much better than \eqref{eq:class number lower bound}, but they possess a great defect in that the implied constants which depend on $\varepsilon$ are \emph{ineffective}. That is, the proofs do not allow the constants to be computed. The ability to compute this implied constant is very important for certain applications, such as the class number problem. We note that if $\psi$ does not have a real zero close to $s = 1$ then we have the stronger bound
\begin{align*}
L(1,\psi) \gg \frac{1}{\log D}
\end{align*}
with an effective implied constant [\textbf{\ref{FI18}}].

Despite the nuisance that exceptional characters represent, there are many situations in which the existence of these characters allows one to prove theorems out of reach of unconditional methods. One can already see hints of this, with exceptional characters being used to control other characters or quantities, in the proof of Siegel's lower bound for $L(1,\psi)$. By way of illustration, if these exceptional characters do exist, then one can derive the existence of twin primes [\textbf{\ref{HB83}}], small primes in arithmetic progressions [\textbf{\ref{FI03}}], and cancellations in sums of Kloosterman sums over the primes [\textbf{\ref{DM19}}]. We refer the reader to the survey articles of Iwaniec [\textbf{\ref{Iw06}}] and Friedlander and Iwaniec [\textbf{\ref{FI17}}] for further information.

In the present paper we study how the hypothetical existence of Landau-Siegel zeros
may influence (increase) the percentages from \eqref{eq:BM k deriv lower bound} in certain corresponding
ranges. We are especially interested to see if there exists a positive integer $k$ for
which the corresponding percentage is pushed to 100\%. As we shall see below, this is indeed the case. Moreover, a rather weak assumption on the
corresponding Landau-Siegel zero, respectively on $L(1,\psi)$, of the form
$L(1,\psi) \le (\log D)^{-r}$ for some suitable fixed $r$, will be sufficient.

For $k = 0$ we show, under such assumption and in the corresponding range for $q$, that
for at least 50\% of the primitive Dirichlet characters $\chi$ modulo $q$ one has $L(1/2,\chi) \ne 0$. 
Thus, this matches the percentage known under the assumption of the Generalized Riemann Hypothesis.
More precisely, one
has the following result.

\begin{theorem}\label{mthm1}
Let $C > 300$ be a fixed real number. Let $D$ be a positive, squarefree, fundamental discriminant, and let $\psi$ be its associated real primitive character. Assume that $\psi$ is even. Then for any $\varepsilon > 0$ and any prime $q$ satisfying
\begin{align*}
D^{300} \leq q \leq D^{C}
\end{align*}
we have
\begin{align*}
\frac{1}{\phi(q)}\sideset{}{^*}\sum_{\substack{\chi (\emph{mod}\ q) \\ L\left( 1/2,\chi \right) \neq 0}} 1&\geq \frac12 + O_{\eps,C}  \big(L(1,\psi)^{1/2}(\log q)^{25/2+\eps}\big)\\
&\qquad\qquad+O_{\eps,C}  \big(L(1,\psi)(\log q)^{25+\eps}\big)+O_{\eps,C}  \big((\log q)^{-1/2+\eps}\big).
\end{align*}
\end{theorem}

We work under the assumptions that $D$ is squarefree and $\psi$ is even, but one could easily extend our arguments to handle other cases.

Next, we prove under similar assumptions that the percentage increases to 100\%
already when $k=1$.
In other words, the proportion of primitive Dirichlet characters $\chi$ for which $L(s,\chi)$ has a multiple
zero at 1/2 is {\it zero}. By way of comparison, under the Generalized Riemann Hypothesis it is only known that this proportion is $\le 1/4$.
We will prove the following result.

\begin{theorem}\label{mthm2} 
Let $C, D$, $\psi$ and $q$ be as in the statement of Theorem \ref{mthm1}. Then for any $\varepsilon > 0$ we have
\begin{align*}
\frac{1}{\phi(q)} \sideset{}{^*}\sum_{\substack{\chi(\textup{mod }q) \\ L(1/2,\chi)= L'(1/2,\chi) = 0}} 1 \ll_{\eps,C}  L(1,\psi)^{1/2}(\log q)^{25/2+\eps}+L(1,\psi)(\log q)^{25+\eps}+(\log q)^{-1/2+\eps}.
\end{align*}
\end{theorem}
We end the introduction by noting that our work may be interpreted, in some respects, as a $q$-aspect analogue of the beautiful work of Conrey and Iwaniec [\textbf{\ref{CI}}].

\section{Setup}

Throughout we let $D>1$ be a squarefree fundamental discriminant and let $\psi$ be a primitive even quadratic Dirichlet character modulo $D$. (These restrictions are merely for technical convenience. One could prove a Voronoi summation formula like that in the Appendix for fundamental discriminants $D \equiv 0 \pmod{4}$, say. Then one needs to take extra care since $D$ is not squarefree at two.) We think of $D$ as being quite small compared to $q$, which is a large prime. The work is unconditional. With more effort we could reduce the exponent of $\log q$ and extend the range of $q$ in Theorem \ref{mthm1} and Theorem \ref{mthm2}, but in the interest of clarity we have chosen not to do so.

We treat the characters $\chi$ modulo $q$ according to their parity to ensure that the $L$-functions under consideration have the same gamma factors in their functional equations. In this work we shall deal exclusively with even Dirichlet characters, but our arguments go through identically for odd characters.

We are interested in the values of $L_\chi(1/2)$, where
\begin{align*}
L_\chi(s):=L\left(s,\chi \right)L\left(s,\chi \psi \right).
\end{align*}
The reason for studying this product of $L$-values, instead of just $L(1/2,\chi)$, is that we have
\begin{align*}
L_\chi(s) = \sum_{n\geq1} \frac{(1\star\psi)(n)\chi(n)}{n^s},
\end{align*}
and the function $(1\star\psi)(n)$ vanishes very often if $\psi$ is exceptional. An elementary calculation [\textbf{\ref{IK}}; (22.109)] shows that
\begin{align}\label{eq: 1 star psi asymp}
\sum_{n \leq x} \frac{(1\star\psi)(n)}{n} = L(1,\psi) \left( \log x + \gamma\right) + L'(1,\psi) + O \big(x^{-1/2}D^{1/4}\log x\big),
\end{align}
and therefore
\begin{align}\label{eq: lacunarity}
\sum_{D^{2} < n \leq x} \frac{(1\star\psi)(n)}{n} \leq L(1,\psi) \log x.
\end{align}
(Note that we have used here the lower bound \eqref{eq:class number lower bound}.) If $\log x \asymp \log q \asymp \log D$ and $\psi$ is an exceptional character, then this quantity is small. We expect to be able to use this (conjecturally non-existent) lacunarity to great effect.

We shall use the mollification method to bound the proportion of characters $\chi$ for which $L(1/2,\chi) \neq 0$, or $L^{(k)}(1/2,\chi) \neq 0$ in general. By the Cauchy-Schwarz inequality we have
\begin{align}\label{eq:cauchy schwarz setup}
\frac{\left|\sum_{\chi (\text{mod}\ q)}^+ L_\chi( 1/2)M(\chi) \right|^2}{\sum_{\chi (\text{mod}\ q)}^+ \left|L_\chi(1/2)M(\chi) \right|^2} &\leq \sideset{}{^+}\sum_{\substack{\chi (\text{mod}\ q) \\ L\chi( 1/2) \neq 0}} 1\leq \sideset{}{^+}\sum_{\substack{\chi (\text{mod}\ q) \\ L\left( 1/2,\chi \right) \neq 0}} 1,
\end{align}
where $\sum^+$ denotes summation over all primitive even characters and $M(\chi)$ is a mollifier to dampen large values of $L_\chi( 1/2)$. As
\begin{align*}
L_\chi(s)^{-1} =  \sum_{n\geq1} \frac{\big(\mu\star(\mu\psi)\big)(n)\chi(n)}{n^s},
\end{align*}
our mollifier is chosen to be
\begin{align}\label{mollifier}
M(\chi): = \sum_{\substack{a \leq X\\D\nmid a}} \frac{ \rho(a)\chi(a)}{\sqrt{a}},
\end{align}
where 
\[
\rho(a)=\big(\mu\star(\mu\psi)\big)(a).
\]
Here $X=q^\kappa$ for some $\kappa>0$, and by lacunarity we anticipate that even a small $\kappa\ $ should be sufficient. Note that to make our off-diagonal term analysis slightly easier we have removed the terms $D|a$ (see also [\textbf{\ref{CI}}; p. 267]). It is easy to see that $\rho(a)$ is the multiplicative function supported on cubefree integers defined by
\begin{align*}
\rho(p) = -\big(1+\psi(p)\big)\qquad\text{and} \qquad \rho(p^2) = \psi(p).
\end{align*}

Since $D$ is small compared to $q$, the mollified first and second moments in \eqref{eq:cauchy schwarz setup} are more similar to the mollified second and fourth moments of $L(1/2,\chi)$. In contrast to recent works on the mollified moments of $L(1/2,\chi)$ (e.g. [\textbf{\ref{Y}}, \textbf{\ref{H}}, \textbf{\ref{BFKMM}}, \textbf{\ref{Z}}, \textbf{\ref{BPRZ}}]), however, we do not appeal to the spectral theory of automorphic forms. The reason for this is that, instead of utilizing an approximate function equation for $\left|L_\chi(1/2) \right|^2$, we obtain an expression for $\left|L_\chi(1/2) \right|^2$ by squaring the approximate functional equation for $L_\chi(1/2)$. This ensures that the ranges of the summation variables in our expression for $\left|L_\chi(1/2) \right|^2$ are short and almost the same size, at the cost of introducing root numbers. In our exceptional setting the root numbers play almost no role, and we may eliminate them with the Cauchy-Schwarz inequality. The off-diagonal terms are then handled with the delta method of Duke, Friedlander and Iwaniec [\textbf{\ref{DFI}}] and an appeal to the Weil bound for Kloosterman sums.

\begin{remark}
\emph{Throughout the paper we use $\varepsilon$ to denote an arbitrarily small positive number, and $C$ to denote a fixed large positive constant. These values may change from one line to the next.}
\end{remark}

\section{Various lemmas}

\begin{lemma}[Orthogonality]\label{ortho}
For $(mn,q)=1$ we have
\[
\sideset{}{^+}\sum_{\chi(\emph{mod}\ q)} \chi(m)\overline{\chi}(n)=\frac 12\sum_{\substack{d|q\\d|(m\pm n)}}\mu\Big(\frac qd\Big)\phi(d)=\mathbf{1}_{q|(m\pm n)}\frac{\phi(q)}{2}-1.
\]
\end{lemma}

For $\chi$ a primitive even character modulo $q$, the Dirichlet $L$-function $L(s,\chi)$ satisfies the functional equation
\begin{align*}
\Lambda(s,\chi)&:=\Big(\frac{q}{\pi}\Big)^{s/2}\Gamma\Big(\frac{s}{2}\Big)L(s,\chi)\\
& =\epsilon(\chi)\Lambda(1-s,\overline\chi),
\end{align*}
where
$$ \epsilon(\chi)= \frac{\tau(\chi)}{q^{1/2}}= \frac{1}{q^{1/2}}\ \sideset{}{^*} \sum_{a(\text{mod}\ \! q)} \chi(a) e\Big(\frac{a}{q}\Big).$$ Note that $|\epsilon(\chi)|=1$. A similar functional equation holds for $L(s,\chi\psi)$, and hence
\begin{align}\label{afe}
\Lambda_\chi(s)&:=\Lambda(s,\chi)\Lambda(s,\chi\psi)\nonumber\\
& =\epsilon(\chi)\epsilon(\chi\psi)\Lambda_{\overline\chi}(1-s).
\end{align}

We need an approximate functional equation to represent the $L$-values.

\begin{lemma}
Let
\begin{equation}\label{formulaV}
V_1(\alpha,x)=\frac{1}{2\pi i}\int_{(1)}\frac{\Gamma(\frac{1/2+ \alpha+s}{2})^2}{\Gamma(\frac{1/2+\alpha}{2})^2}x^{-s}\frac{ds}{s}\quad\text{and}\quad V_2(\alpha,x)=\frac{1}{2\pi i}\int_{(1)}\frac{\Gamma(\frac{1/2- \alpha+s}{2})^2}{\Gamma(\frac{1/2+\alpha}{2})^2}x^{-s}\frac{ds}{s}.
\end{equation}
Then we have
\begin{align}\label{formulaL}
L_\chi(\tfrac{1}{2}+\alpha)=\sum_{n}\frac{(1\star\psi)(n)\chi(n)}{n^{1/2+\alpha}}V_{1}\Big(\alpha,\frac{n}{Q}\Big)+\epsilon(\chi)\epsilon(\chi\psi)Q^{-2\alpha}\sum_{n}\frac{(1\star\psi)(n)\overline{\chi}(n)}{n^{1/2-\alpha}}V_{2}\Big(\alpha,\frac{n}{Q}\Big),
\end{align}
where
\[
Q=\frac{q\sqrt{D}}{\pi}.
\]
\end{lemma}
\begin{proof}
Using Cauchy's theorem we have
\[
\frac{1}{2\pi i}\int_{(1)}\frac{\Lambda_\chi(1/2+\alpha+s)}{\Gamma(\frac{1/2+\alpha}{2})^2}\frac{ds}{s}=\text{Res}_{s=0}+\frac{1}{2\pi i}\int_{(-1)}\frac{\Lambda_\chi(1/2+\alpha+s)}{\Gamma(\frac{1/2+\alpha}{2})^2}\frac{ds}{s}.
\]
Clearly the residue at $s=0$ is
\begin{displaymath}
Q^{1/2+\alpha}L_\chi(\tfrac{1}{2}+\alpha).
\end{displaymath}
By a change of variables $s\leftrightarrow-s$ and using \eqref{afe}, we then obtain
\begin{displaymath}
Q^{1/2+\alpha}L_\chi(\tfrac{1}{2}+\alpha)=\frac{1}{2\pi i}\int_{(1)}\frac{\Lambda_\chi(1/2+\alpha+s)}{\Gamma(\frac{1/2+\alpha}{2})^2}\frac{ds}{s}+\frac{\epsilon(\chi)\epsilon(\chi\psi)}{2\pi i}\int_{(1)}\frac{\Lambda_{\overline{\chi}}(1/2-\alpha+s)}{\Gamma(\frac{1/2+\alpha}{2})^2}\frac{ds}{s}.
\end{displaymath}
The lemma now follows by writing $\Lambda_\chi$ in terms of Dirichlet series and then integrating term-by-term.
\end{proof}

For $i=1,2$ we denote
\[
V_i(x)=V_i(0,x)\qquad\text{and}\qquad \partial V_i(x)=\partial_{\alpha}V_i(\alpha,x)\big|_{\alpha=0}.
\]
Observe that $V_1(x)=V_2(x)$. Taking derivative of \eqref{formulaL} with respect to $\alpha$ and setting $\alpha=0$ we obtain the following corollary.

\begin{corollary}
Let
\[
W_{1}(x)=\Big(\frac12-\frac{\log x}{2\log Q}\Big)V_1(x)+\frac{\partial V_1(x)}{2\log Q}\quad\text{and}\quad W_{2}(x)=\Big(\frac12+\frac{\log x}{2\log Q}\Big)V_2(x)+\frac{\partial V_2(x)}{2\log Q}.
\]
Then we have
\begin{equation}\label{formulaL'}
L_\chi(\tfrac{1}{2})+\frac{1}{2\log Q}L_\chi'(\tfrac12)=\sum_{n}\frac{(1\star\psi)(n)\chi(n)}{\sqrt{n}}W_{1}\Big(\frac{n}{Q}\Big)+\epsilon(\chi)\epsilon(\chi\psi)\sum_{n}\frac{(1\star\psi)(n)\overline{\chi}(n)}{\sqrt{n}}W_{2}\Big(\frac{n}{Q}\Big).
\end{equation}
\end{corollary}

We collect the properties of the functions $V_i$ and $W_i$, $i=1,2$, in the following lemma.

\begin{lemma}\label{boundVW}
The Mellin transforms $\widetilde{V_i}(s)$ and $\widetilde{W_i}(s)$ decay rapidly as $|s|\rightarrow\infty$ in the half plane $\emph{Re}(s)>-1/2+\varepsilon$ and have only a pole at $s=0$. For $s$ sufficiently small we can write
\begin{align*}
\widetilde{V_i}(s) = \frac{c_{-2}}{s^2} + \frac{c_{-1}}{s} + H(s),
\end{align*}
where $H$ is holomorphic, $|c_{-2}| \ll (\log Q)^{-1}$, and $|c_{-1}| \ll 1$. The same is true for $\widetilde{W_i}(s)$. Also,
\[
x^jV_i^{(j)}(x),\,x^jW_i^{(j)}(x)\ll_{j,C} (1+x)^{-C}
\]
for any fixed $j\geq0$ and $C > 0$. Furthermore, we have
\[
V_i(x)=1+O_\varepsilon\big(x^{1/2-\varepsilon}\big),
\]
\[
W_1(x)=\frac12-\frac{\log x}{2\log Q}+O_\varepsilon\big(x^{1/2-\varepsilon}\big)
\]
and
\[
W_2(x)=\frac12+\frac{\log x}{2\log Q}-\frac{2\Gamma'(1/4)}{\Gamma(1/4)}\frac{1}{\log Q}+O_\varepsilon\big(x^{1/2-\varepsilon}\big).
\]
\end{lemma}
\begin{proof}
The first statement can be proved by direct calculations. For instance, it is easy to see from the definition of the Mellin transform that
\[
\widetilde{V_i}(s)=\frac{1}{\Gamma(1/4)^2s}\Gamma\Big(\frac{1}{4}+\frac{s}{2}\Big)^2
\]
and
\begin{align*}
\widetilde{W_1}(s)&=\frac{1}{2\Gamma(1/4)^2}\Gamma\Big(\frac{1}{4}+\frac{s}{2}\Big)^2\bigg(\Big(1-\frac{\Gamma'(1/4)}{\Gamma(1/4)\log Q}\Big)\frac{1}{s}+\frac{1}{(\log Q)s^2}\bigg).
\end{align*}

The second statement is obtained by moving the lines of integration in \eqref{formulaV} to $\text{Re}(s)=C$. For the last statement we move the contours to $\text{Re}(s)=-1/2+\varepsilon$ getting
\[
V_i(x)=1+O_\varepsilon\big(x^{1/2-\varepsilon}\big),
\]
\[
\partial V_1(x)=O_\varepsilon\big(x^{1/2-\varepsilon}\big)\qquad\text{and}\qquad \partial V_2(x)=-\frac{2\Gamma'(1/4)}{\Gamma(1/4)}+O_\varepsilon\big(x^{1/2-\varepsilon}\big).
\]The estimates for $W_1$ and $W_2$ then follow.
\end{proof}

\section{Main proposition and deduction of Theorem \ref{mthm1} and Theorem \ref{mthm2}}

We shall prove Theorem \ref{mthm1} and Theorem \ref{mthm2} by considering the mollified first and second moments of $L_\chi(1/2)$ and $L_\chi(1/2)+L_\chi'(1/2)/2\log Q$, respectively. For this we first rewrite \eqref{formulaL} and \eqref{formulaL'} in the following form:
\begin{align*}
L_\chi(\tfrac{1}{2})&=V_{1}\Big(\frac{1}{Q}\Big)+\epsilon(\chi)\epsilon(\chi\psi)V_2\Big(\frac{1}{Q}\Big)\\
&\qquad\qquad +\sum_{n>1}\frac{(1\star\psi)(n)\chi(n)}{\sqrt{n}}V_{1}\Big(\frac{n}{Q}\Big)+\epsilon(\chi)\epsilon(\chi\psi)\overline{\sum_{n>1}\frac{(1\star\psi)(n)\chi(n)}{\sqrt{n}}V_{2}\Big(\frac{n}{Q}\Big)}
\end{align*}
and
\begin{align*}
&L_\chi(\tfrac{1}{2})+\frac{1}{2\log Q}L_\chi'(\tfrac{1}{2})=W_{1}\Big(\frac{1}{Q}\Big)+\epsilon(\chi)\epsilon(\chi\psi)W_2\Big(\frac{1}{Q}\Big)\\
&\qquad\qquad+\sum_{n>1}\frac{(1\star\psi)(n)\chi(n)}{\sqrt{n}}W_{1}\Big(\frac{n}{Q}\Big)+\epsilon(\chi)\epsilon(\chi\psi)\overline{\sum_{n>1}\frac{(1\star\psi)(n)\chi(n)}{\sqrt{n}}W_{2}\Big(\frac{n}{Q}\Big)}.
\end{align*}
We now let
\[
X=D^{20}
\]
and multiply the above expressions with the mollifier \eqref{mollifier} to obtain
\begin{equation}\label{mollifier1}
L_\chi(\tfrac12)M(\chi)=V_1\Big(\frac{1}{Q}\Big)+\epsilon(\chi)\epsilon(\chi\psi)V_2\Big(\frac{1}{Q}\Big)+O\big(|B_1(\chi)|\big)+O\big(|B_2(\chi)|\big)
\end{equation}
and
\begin{equation}\label{mollifier2}
\Big(L_\chi(\tfrac{1}{2})+\frac{1}{2\log Q}L_\chi'(\tfrac{1}{2})\Big)M(\chi)=W_1\Big(\frac{1}{Q}\Big)+\epsilon(\chi)\epsilon(\chi\psi)W_2\Big(\frac{1}{Q}\Big)+O\big(|C_1(\chi)|\big)+O\big(|C_2(\chi)|\big),
\end{equation}
where the Dirichlet polynomials  $B_i(\chi)$ and $C_i(\chi)$, $i=1,2$, are given by
\begin{align*}
&B_i(\chi)=\sum_{\substack{a\leq X\\D\nmid a\\ an>1}}\frac{\rho(a)(1\star\psi)(n)\chi(an)}{\sqrt{an}}V_i\Big(\frac{n}{Q}\Big),\qquad C_i(\chi)=\sum_{\substack{a\leq X\\D\nmid a\\ an>1}}\frac{\rho(a)(1\star\psi)(n)\chi(an)}{\sqrt{an}}W_i\Big(\frac{n}{Q}\Big).
\end{align*}
Our next proposition will show that the second moments of $B_i(\chi)$ and $C_i(\chi)$ are negligible for certain ranges of $q$ and $D$.

Let $V:\mathbb{R}_+\rightarrow\mathbb{C}$ be such that the Mellin transform $\widetilde{V}(s)$ decays rapidly as $|s|\rightarrow\infty$ for $\text{Re}(s)>-1/2+\varepsilon$ and has only a pole at $s=0$ with principal part
\[
\frac{v_2}{(\log Q)s^2}+\frac{v_1}{s}.
\] Let 
\[
B(\chi)=\sum_{\substack{a\leq X\\D\nmid a\\ an>1}}\frac{\rho(a)(1\star\psi)(n)\chi(an)}{\sqrt{an}}V\Big(\frac{n}{Q}\Big).
\]
In the remaining sections, we shall prove the following key proposition.

\begin{proposition}\label{mainprop}
Let $C>300$ be a fixed real number. Then for any prime $q$ satisfying
\[
D^{300}\leq q\leq D^C
\] 
we have
\[
\sideset{}{^+}\sum_{\chi (\emph{mod}\ q)}\big|B(\chi)\big|^2\ll_{\eps,C} L(1,\psi)q(\log q)^{25+\eps}+q(\log q)^{-1+\eps}.
\]
\end{proposition}

With this result, we can now prove Theorem \ref{mthm1} and Theorem \ref{mthm2}.

\begin{proof}[Proof of Theorem \ref{mthm1}]
Applying Proposition \ref{mainprop} to $B_1(\chi)$ and $B_2(\chi)$ we get
\begin{align}\label{firstsquare}
\sideset{}{^+}\sum_{\chi (\text{mod}\ q)}\big|B_i(\chi)\big|^2\ll_{\eps,C} L(1,\psi)q(\log q)^{25+\eps}+q(\log q)^{-1+\eps}
\end{align}
for $i=1,2$. The Cauchy-Schwarz inequality then implies
\begin{equation}\label{first0}
\sideset{}{^+}\sum_{\chi (\text{mod}\ q)}|B_i(\chi)\big|\ll_{\eps,C} L(1,\psi)^{1/2}q(\log q)^{25/2+\eps}+q(\log q)^{-1/2+\eps}.
\end{equation}
In view of \eqref{mollifier1} we have
\begin{align}\label{first1}
\sideset{}{^+}\sum_{\chi (\text{mod}\ q)} L_\chi(\tfrac12)M(\chi)&=\phi^+(q)V_1\Big(\frac{1}{Q}\Big)+V_2\Big(\frac{1}{Q}\Big)\ \sideset{}{^+}\sum_{\chi (\text{mod}\ q)}\epsilon(\chi)\epsilon(\chi\psi)\\
&\qquad\qquad+O_{\eps,C} \big(L(1,\psi)^{1/2}q(\log q)^{25/2+\eps}\big)+O_{\eps,C} \big(q(\log q)^{-1/2+\eps}\big),\nonumber
\end{align}
where $\phi^+(q)$ is the number of primitive even characters modulo $q$. 

Since $(q,D) = 1$ ($q$ is a prime greater than $D$), an elementary calculation with the Chinese remainder theorem shows that
\begin{align*}
\epsilon(\chi \psi) = \chi(D) \psi(q) \epsilon(\chi) \epsilon(\psi).
\end{align*}
So, using the definition of $\epsilon(\chi)$, 
\[
\sideset{}{^+}\sum_{\chi (\text{mod}\ q)}\epsilon(\chi)\epsilon(\chi\psi)=\frac{\psi(q) \epsilon(\psi)}{q}\ \sideset{}{^*}\sum_{a,b (\text{mod}\ q)}e\Big(\frac{a+b}{q}\Big) \sideset{}{^+}\sum_{\chi (\text{mod}\ q)}\chi(Dab) ,
\]
which, by the orthogonality in Lemma \ref{ortho}, is equal to
\begin{align*}
&\frac{\psi(q) \epsilon(\psi)\phi(q)}{2q}\ \sideset{}{^*}\sum_{\substack{a,b (\text{mod}\ q)\\Dab \equiv \pm  1 (\text{mod}\ q)}}e\Big(\frac{a+b}{q}\Big)-\frac{\psi(q) \epsilon(\psi)}{q}\ \sideset{}{^*}\sum_{a,b (\text{mod}\ q)}e\Big(\frac{a+b}{q}\Big)\\
&\qquad=\frac{\psi(q) \epsilon(\psi)\phi(q)}{2q}\big(S(1,\overline{D};q)+S(1,-\overline{D};q)\big)-\frac{\psi(q) \epsilon(\psi)}{q}.
\end{align*}
Hence, by the Weil bound,
\begin{equation}\label{first2}
\sideset{}{^+}\sum_{\chi (\text{mod}\ q)}\epsilon(\chi)\epsilon(\chi\psi)\ll q^{1/2}.
\end{equation}
From \eqref{first1}, \eqref{first2} and Lemma \ref{boundVW} we then get
\begin{align*}
\sideset{}{^+}\sum_{\chi (\text{mod}\ q)} L_\chi(\tfrac12)M(\chi)&=\phi^+(q)+O_{\eps,C} \big(L(1,\psi)^{1/2}q(\log q)^{25/2+\eps}\big)+O_{\eps,C} \big(q(\log q)^{-1/2+\eps}\big).
\end{align*}

For the second moment, we first square \eqref{mollifier1} to obtain
\begin{align*}
\big|L_\chi(\tfrac12)M(\chi)\big|^2&=\bigg|V_1\Big(\frac{1}{Q}\Big)\bigg|^2+\bigg|V_2\Big(\frac{1}{Q}\Big)\bigg|^2+2\text{Re}\bigg(\epsilon(\chi)\epsilon(\chi\psi)V_1\Big(\frac{1}{Q}\Big)V_2\Big(\frac{1}{Q}\Big)\bigg)\\
&\qquad\qquad+O\big(|B_1(\chi)|+|B_1(\chi)|^2\big)+O\big(|B_2(\chi)|+|B_2(\chi)|^2\big).
\end{align*}
Using \eqref{firstsquare}, \eqref{first0}, \eqref{first2} and Lemma \ref{boundVW} we get
\begin{align*}
\sideset{}{^+}\sum_{\chi (\text{mod}\ q)} \big|L_\chi(\tfrac12)M(\chi)\big|^2&=2\phi^+(q)+O_{\eps,C} \big(L(1,\psi)^{1/2}q(\log q)^{25/2+\eps}\big)\\
&\qquad\qquad+O_{\eps,C} \big(L(1,\psi)q(\log q)^{25+\eps}\big)+O_{\eps,C} \big(q(\log q)^{-1/2+\eps}\big).
\end{align*}
The Cauchy-Schwarz inequality then leads to
\begin{align*}
\sideset{}{^+}\sum_{\substack{\chi (\text{mod}\ q) \\ L\left( 1/2,\chi \right) \neq 0}} 1\geq \sideset{}{^+}\sum_{\substack{\chi (\text{mod}\ q) \\ L\chi( 1/2) \neq 0}} 1&\geq \frac{\left|\sum_{\chi (\text{mod}\ q)}^+ L_\chi( 1/2)M(\chi) \right|^2}{\sum_{\chi (\text{mod}\ q)}^+ \left|L_\chi(1/2)M(\chi) \right|^2}\\
&=\frac{\phi^+(q)}{2}+O_{\eps,C} \big(L(1,\psi)^{1/2}q(\log q)^{25/2+\eps}\big)\\
&\qquad\qquad+O_{\eps,C} \big(L(1,\psi)q(\log q)^{25+\eps}\big)+O_{\eps,C} \big(q(\log q)^{-1/2+\eps}\big),
\end{align*}
which completes the proof.
\end{proof}

\begin{proof}[Proof of Theorem \ref{mthm2}]
Using \eqref{mollifier2}, we can see that the exact argument above can be applied to evaluate the first and second moments of 
\[
\Big(L_\chi(\tfrac12)+\frac{1}{2\log Q}L_\chi'(\tfrac 12)\Big)M(\chi).
\]
The only difference here, as in Lemma \ref{boundVW}, is
\[
W_2\Big(\frac 1L\Big)=O\big((\log Q)^{-1}\big),
\]
while in the previous case we have
\[
V_2\Big(\frac 1L\Big)=1+O_\varepsilon\big(Q^{-1/2+\varepsilon}\big).
\]
So, as before,
\begin{align*}
\sideset{}{^+}\sum_{\chi (\text{mod}\ q)} \Big(L_\chi(\tfrac12)+\frac{1}{2\log Q}L_\chi'(\tfrac 12)\Big)M(\chi)&=\phi^+(q)+O_{\eps,C} \big(L(1,\psi)^{1/2}q(\log q)^{25/2+\eps}\big)\\
&\qquad\qquad+O_{\eps,C} \big(q(\log q)^{-1/2+\eps}\big),
\end{align*}
but
\begin{align*}
&\sideset{}{^+}\sum_{\chi (\text{mod}\ q)} \bigg|\Big(L_\chi(\tfrac12)+\frac{1}{2\log Q}L_\chi'(\tfrac 12)\Big)M(\chi)\bigg|^2=\phi^+(q)+O_{\eps,C} \big(L(1,\psi)^{1/2}q(\log q)^{25/2+\eps}\big)\\
&\qquad\qquad+O_{\eps,C} \big(L(1,\psi)q(\log q)^{25+\eps}\big)+O_{\eps,C} \big(q(\log q)^{-1/2+\eps}\big).
\end{align*}
Applying the Cauchy-Schwarz inequality we obtain
 \begin{align*}
\sideset{}{^+}\sum_{\substack{\chi (\text{mod}\ q) \\ L_\chi(1/2)+L_\chi'(1/2)/2\log Q \neq 0}} 1&\geq \phi^+(q)+O_{\eps,C} \big(L(1,\psi)^{1/2}q(\log q)^{25/2+\eps}\big)\\
&\qquad\qquad+O_{\eps,C} \big(L(1,\psi)q(\log q)^{25+\eps}\big)+O_{\eps,C} \big(q(\log q)^{-1/2+\eps}\big).
\end{align*}
Observe that
\begin{align*}
\sideset{}{^+}\sum_{\substack{\chi(\textup{mod }q) \\ L(1/2,\chi)= L'(1/2,\chi) = 0}} 1 \leq  \phi^+(q)-\sideset{}{^+}\sum_{\substack{\chi (\text{mod}\ q) \\ L_\chi(1/2)+L_\chi'(1/2)/2\log Q \neq 0}} 1,
\end{align*}
and we obtain the theorem.
\end{proof}

\section{Proof of Proposition \ref{mainprop} - preparations}

We wish to evaluate
\begin{align*}
I &= \sideset{}{^+}\sum_{\chi (\text{mod}\ q)}\bigg|\sum_{\substack{a\leq X\\D\nmid a\\ an>1}}\frac{\rho(a)(1\star\psi)(n)\chi(an)}{\sqrt{an}}V\Big(\frac{n}{Q}\Big)\bigg|^2,
\end{align*}
where $X=D^{20}$. We apply Lemma \ref{ortho} to obtain
\begin{align*}
I &=\frac{\phi(q)}{2}\sum_{\substack{am\equiv \pm bn(\text{mod}\ q) \\a,b\leq X\\D\nmid a,\,D\nmid b\\am,bn>1\\ (mn,q)=1}}\frac{\rho(a)\rho(b)(1\star\psi)(m)(1\star\psi)(n)}{\sqrt{abmn}}V\Big(\frac{m}{Q}\Big)V\Big(\frac{n}{Q}\Big)\\
&\qquad\qquad\qquad\qquad-\sum_{\substack{a,b\leq X\\D\nmid a,\,D\nmid b\\am,bn>1\\ (mn,q)=1}}\frac{\rho(a)\rho(b)(1\star\psi)(m)(1\star\psi)(n)}{\sqrt{abmn}}V\Big(\frac{m}{Q}\Big)V\Big(\frac{n}{Q}\Big).
\end{align*}
Note that the condition $(mn,q)=1$ may be omitted with the cost of an error of size $O_\varepsilon(q^{\varepsilon}\sqrt{D}X)$. So we can write
\begin{align}\label{formulaI}
I&=\frac{\phi(q)}{2}\sum_{\substack{am\equiv \pm bn(\text{mod}\ q) \\a,b\leq X\\D\nmid a,\,D\nmid b\\am,bn>1}}\frac{\rho(a)\rho(b)(1\star\psi)(m)(1\star\psi)(n)}{\sqrt{abmn}}V\Big(\frac{m}{Q}\Big)V\Big(\frac{n}{Q}\Big)\nonumber\\
&\qquad\qquad\qquad\qquad-\bigg(\sum_{\substack{a\leq X\\D\nmid a\\am>1}}\frac{\rho(a)(1\star\psi)(m)}{\sqrt{am}}V\Big(\frac{m}{Q}\Big)\bigg)^2+O_\varepsilon\big(q^{\varepsilon}\sqrt{D}X\big)\nonumber\\
&=I^{D}+I^{OD}-\bigg(\sum_{\substack{a\leq X\\D\nmid a\\am>1}}\frac{\rho(a)(1\star\psi)(m)}{\sqrt{am}}V\Big(\frac{m}{Q}\Big)\bigg)^2+O_\varepsilon\big(q^{\varepsilon}\sqrt{D}X\big),
\end{align}
where $I^{D}$ and $I^{OD}$ are the contributions from the diagonal terms $am=bn$ and the off-diagonal terms $am\ne bn$ in the first sum, respectively. We shall estimate $I^{D}$ in Section \ref{diagonal1} and $I^{OD}$ in Section \ref{offdiagonal1}. We finish this section with the evaluation of the third term in \eqref{formulaI}. 

We have
\begin{equation}\label{1000}
\sum_{\substack{a\leq X\\D\nmid a\\am>1}}\frac{\rho(a)(1\star\psi)(m)}{\sqrt{am}}V\Big(\frac{m}{Q}\Big)=\sum_{\substack{a\leq X\\D\nmid a}}\frac{\rho(a)}{\sqrt{a}}\sum_{m\geq1}\frac{(1\star\psi)(m)}{\sqrt{m}}V\Big(\frac{m}{Q}\Big)-V\Big(\frac{1}{Q}\Big).
\end{equation}
Using Mellin inversion,
\begin{align*}
&\sum_{m\geq 1} \frac{(1\star\psi)(m)}{\sqrt{m}}V\Big(\frac{m}{Q}\Big)=\frac{1}{2\pi i}\int_{(1)}\widetilde{V}(u)Q^u\zeta(\tfrac12+u)L(\tfrac12+u,\psi)du,
\end{align*}
and by moving the line of integration to $\text{Re}(u)=\eps$, we see that it is equal to
\[
L(1,\psi)\widetilde{V}(\tfrac12)Q^{1/2}+O_\eps(Q^\eps).
\]
So the expression in \eqref{1000} is $$L(1,\psi)\widetilde{V}(\tfrac12)Q^{1/2}\sum_{\substack{a\leq X\\D\nmid a}}\frac{\rho(a)}{\sqrt{a}}+O_\eps\big(q^{\eps}X^{1/2}\big),$$ and hence
\begin{align}\label{mainformulaI}
I=I^{D}+I^{OD}-L(1,\psi)^2\widetilde{V}(\tfrac12)^2Q\sum_{\substack{a,b\leq X\\D\nmid a,\,D\nmid b}}\frac{\rho(a)\rho(b)}{\sqrt{ab}}+O_\varepsilon\big(L(1,\psi)q^{1/2+\varepsilon}D^{1/4}X\big).
\end{align}

\section{The diagonal $I^{D}$}\label{diagonal1}

\subsection{Preparations}

We have
\[
I^D=\frac{\phi(q)}{2}\sum_{\substack{am= bn\\a,b\leq X\\D\nmid a,\,D\nmid b}}\frac{\rho(a)\rho(b)(1\star\psi)(m)(1\star\psi)(n)}{\sqrt{abmn}}V\Big(\frac{m}{Q}\Big)V\Big(\frac{n}{Q}\Big)-\frac{\phi(q)}{2}V\Big(\frac{1}{Q}\Big)^2.
\]
We can remove the conditions $D\nmid a$ and $D\nmid b$ at the cost of an error of size $O_\eps(q^{1+\eps}D^{-1})$. For the condition $am=bn$, we may change the variables $a\rightarrow ha$, $b\rightarrow hb$, $m\rightarrow bn$ and $n\rightarrow an$ with $(a,b)=1$. Therefore
\begin{align*}
I^D &= \frac{\phi(q)}{2}\sum_{h \leq X} \frac{1}{h}\sum_{\substack{a,b \leq X/h \\ (a,b)=1}} \frac{\rho(ha) \rho(hb)}{ab}\sum_{n\geq 1} \frac{(1\star\psi)(an)(1\star\psi)(bn)}{n}V\Big(\frac{an}{Q}\Big)V\Big(\frac{bn}{Q}\Big)\\
&\qquad\qquad-\frac{\phi(q)}{2}V\Big(\frac{1}{Q}\Big)^2+O_\eps\big(q^{1+\eps}D^{-1}\big).
\end{align*}

We first evaluate the sum over $n$. Writing $V$ in terms of its Mellin transform we get
\begin{align*}
&\sum_{n\geq 1} \frac{(1\star\psi)(an)(1\star\psi)(bn)}{n}V\Big(\frac{an}{Q}\Big)V\Big(\frac{bn}{Q}\Big)\\
&\qquad\qquad=\frac{1}{(2\pi i)^2}\int_{(c)}\int_{(c)}\widetilde{V}(u)\widetilde{V}(v)\Big(\frac{Q}{a}\Big)^u\Big(\frac{Q}{b}\Big)^{v}\sum_{n\geq 1} \frac{(1\star\psi)(an)(1\star\psi)(bn)}{n^{1+u+v}}dudv
\end{align*}
with $c=1/\log q$. By multiplicativity,
\[
\sum_{n\geq 1} \frac{(1\star\psi)(an)(1\star\psi)(bn)}{n^{s}}=f_a(s)f_b(s)\prod_p\bigg(1+\sum_{j\geq1}\frac{(1\star\psi)(p^j)^2}{p^{js}}\bigg),
\]
where
\[
f_a(s)=\prod_{p^{a_p}||a}\bigg((1\star\psi)(p^{a_p})+\sum_{j\geq1}\frac{(1\star\psi)(p^{a_p+j})(1\star\psi)(p^j)}{p^{js}}\bigg)\bigg(1+\sum_{j\geq1}\frac{(1\star\psi)(p^j)^2}{p^{js}}\bigg)^{-1}.
\]
Also,
\begin{align*}
\prod_p\bigg(1+\sum_{j\geq1}\frac{(1\star\psi)(p^j)^2}{p^{js}}\bigg)&=\prod_{p|D}\bigg(1-\frac{1}{p^{s}}\bigg)^{-1}\prod_{\psi(p)=-1}\bigg(1-\frac{1}{p^{2s}}\bigg)^{-1}\prod_{\psi(p)=1}\bigg(1+\frac{1}{p^{s}}\bigg)\bigg(1-\frac{1}{p^{s}}\bigg)^{-3}\\
&=H_D(s)\prod_{\psi(p)=1}\bigg(1-\frac{1}{p^{s}}\bigg)^{-4}=H_D(s)\sum_{\substack{n\geq1\\p|n\rightarrow \psi(p)=1}}\frac{\tau_4(n)}{n^s},
\end{align*}
say, where
\[
H_D(s)=\prod_{p|D}\bigg(1-\frac{1}{p^{s}}\bigg)^{-1}\prod_{\psi(p)=-1}\bigg(1-\frac{1}{p^{2s}}\bigg)^{-1}\prod_{\psi(p)=1}\bigg(1-\frac{1}{p^{2s}}\bigg).
\]
So
\begin{align*}
&\sum_{n\geq 1} \frac{(1\star\psi)(an)(1\star\psi)(bn)}{n}V\Big(\frac{an}{Q}\Big)V\Big(\frac{bn}{Q}\Big)=\frac{1}{(2\pi i)^2}\int_{(c)}\int_{(c)}\widetilde{V}(u)\widetilde{V}(v)\Big(\frac{Q}{a}\Big)^u\Big(\frac{Q}{b}\Big)^{v}\\
&\qquad\qquad f_a(1+u+v)f_b(1+u+v)H_D(1+u+v)\sum_{\substack{n\geq1\\p|n\rightarrow \psi(p)=1}}\frac{\tau_4(n)}{n^{1+u+v}}dudv,
\end{align*}
and hence
\begin{align}\label{550}
I^D& =\frac{\phi(q)}{2}\frac{1}{(2\pi i)^2}\int_{(c)}\int_{(c)}\widetilde{V}(u)\widetilde{V}(v)Q^{u+v}J(u,v)dudv-\frac{\phi(q)}{2}V\Big(\frac{1}{Q}\Big)^2+O_\eps\big(q^{1+\eps}D^{-1}\big),
\end{align}
where
\begin{align}\label{fID}
J(u,v)& =H_D(1+u+v)\sum_{\substack{n\geq1\\p|n\rightarrow \psi(p)=1}}\frac{\tau_4(n)}{n^{1+u+v}}\sum_{h \leq X} \frac{1}{h}\sum_{\substack{a,b \leq X/h \\ (a,b)=1}} \frac{\rho(ha) \rho(hb)f_a(1+u+v)f_b(1+u+v)}{a^{1+u}b^{1+v}}.
\end{align}
Note that the function $f_n(s)$ is multiplicative with respect to $n$ and
\begin{align*}
f_{p^j}(s) &= 
\begin{cases}
1   &\text{if }p|D, \\
0 &\text{if }\psi(p) = -1, \ j \text{ odd}, \\
1 &\text{if }\psi(p) = -1, \ j \text{ even}, \\
1 + \frac{j(p^s-1)}{p^s+1} &\text{if }\psi(p) = 1.
\end{cases}
\end{align*}

\subsection{The mollification}

Our general strategy is as follows. For any prime $p$ we have $\psi(p) \in \{-1,0,1\}$. We factor integer variables into pieces where each piece is only divisible by primes $p$ having the same value of $\psi(p)$. The effect of the mollification is then studied. For clarity we perform this factorization piecemeal rather than all at once.

The typical technique to evaluate mollified sums is to work with Mellin transforms and $L$-functions, but our exceptional situation makes this approach difficult. Therefore, we work combinatorially (in ``physical space'') rather than with integrals (in ``frequency space'').

The following lemma will facilitate some of our estimations.

\begin{lemma}\label{lem:divisor twisted by lacunary}
Let $k$ be a positive integer and let $A>4$ be a fixed real number. If $D$ is sufficiently large in terms of $A$ and $k$, then
\begin{align*}
\sum_{D^{4}<n \leq D^A} \frac{\tau(n)^k (1\star \psi)(n)}{n} &\ll_{\varepsilon,A,C,k} L(1,\psi)(\log D)^{2^{k+1}+1+\varepsilon}  + (\log D)^{-C}.
\end{align*}
\end{lemma}
We remark that this bound is worse than expected by a factor of $(\log D)^{2^k+1}$. This is due to our inability to improve upon the trivial bound $L'(1,\psi) \ll (\log D)^2$.
\begin{proof}
Our idea is to factor $n$ into pieces depending on the size of the prime factors of $n$, since the larger prime factors of $n$ contribute comparatively less to $\tau(n)$.

Let $\mathcal{J}$ be the sum we wish to bound. We first make a technical simplification, reducing to the case when $n$ is squarefree. We can factor $n$ as $n = sn'$, where $s$ is squarefull, $n'$ is squarefree, and $(s,n') = 1$. The contribution from $s > D$ is trivially acceptable, so we may assume that $s \leq D$, and therefore $n' > D^{3}$. Writing $n$ for $n'$, we therefore have
\begin{align*}
\mathcal{J} &\ll_{\varepsilon,k} \sum_{D^{3}<n \leq D^A} \frac{\mu^2(n)\tau(n)^k (1\star \psi)(n)}{n}+D^{-1/2+\varepsilon}.
\end{align*}

Let $t = t(k)$ be a large positive constant. By trivial estimation we may suppose that $\omega(n) \leq t \log \log D$, because the contribution from those $n$ with $\omega(n) > t \log \log D$ is trivially $O_{C,k}((\log D)^{-C})$. We then set $D_0 = D^{1/(\log \log D)^2}$, and factor $n = ab$, where $P^+(a) \leq D_0$ and $P^-(b) > D_0$. Since $\omega(n) \leq t \log \log D$ we see that $a \leq D^{o(1)}$, and therefore $b > D^{3-\varepsilon}$. It follows that
\begin{align*}
\mathcal{J} &\ll_{C,k}  \sum_{\substack{a \leq D^A \\ P^+(a) \leq D_0}} \frac{\mu^2(a)\tau(a)^k (1\star \psi)(a)}{a}\sum_{\substack{D^{3-\varepsilon}<b \leq D^A \\ P^-(b) > D_0}} \frac{\mu^2(b)\tau(b)^k (1\star \psi)(b)}{b}+(\log D)^{-C}.
\end{align*}

For $a$ with $P^+(a) \leq D_0$ we are unable to take advantage of the potential lacunarity of $(1\star\psi)(a)$, and therefore we use the trivial bound
\begin{align*}
 \sum_{\substack{a \leq D^A \\ P^+(a) \leq D_0}} \frac{\mu^2(a)\tau(a)^k (1\star \psi)(a)}{a} \leq \sum_{a \leq D^A} \frac{\tau(a)^{k+1}}{a} \ll_{A,k} (\log D)^{2^{k+1}}.
\end{align*}
Now, the typical integer $b$ has $\approx \log \log D^A - \log \log D_0 \approx 2\log \log \log D$ prime factors, and it is rare to have many more prime factors than this. If $\omega(b) \leq \varepsilon \log \log D$, then $\tau(b) \leq (\log D)^{\varepsilon}$, and the contribution from these integers is
\begin{align*}
\ll_{\varepsilon,A,k} L(1,\psi)(\log D)^{2^{k+1}+1+\varepsilon} 
\end{align*}
by \eqref{eq: lacunarity}. On the other hand, the contribution from those with $\omega(b) > \varepsilon \log \log D$ is
\begin{align*}
&\ll_{A,k} (\log D)^{2^{k+1}}\sum_{m > \varepsilon\log\log D}\sum_{\substack{D^{3-\varepsilon} < b \leq D^A \\ P^-(b) > D_0 \\ \omega(b) =m}} \frac{\mu^2(b)\tau(b)^{k+1}}{b}\\
&\leq (\log D)^{2^{k+1}}\sum_{m > \varepsilon\log\log D} \frac{2^{(k+1)m}}{m!} \bigg(\sum_{D_0 < p \leq D^A} \frac{1}{p} \bigg)^m\\
& \ll_{\varepsilon,A,k} (\log D)^{2^{k+1}}\sum_{m > \varepsilon\log\log D} \frac{(2^{k+2}\log\log\log D)^m}{m !}  \\
&\ll_{\varepsilon,A,k} (\log D)^{2^{k+1}}\Big(\frac{2^{k+4}\log\log\log D}{\varepsilon\log\log D} \Big)^{\varepsilon \log \log D} \ll_{\varepsilon,A,C,k} (\log D)^{-C},
\end{align*}
and the result follows.
\end{proof}

We now consider
\[
\sum_{h \leq X} \frac{1}{h}\sum_{\substack{a,b \leq X/h \\ (a,b)=1}} \frac{\rho(ha) \rho(hb)f_a(1+u+v)f_b(1+u+v)}{a^{1+u}b^{1+v}}.
\]
We begin with those primes such that $\psi(p) = 0$, i.e. $p|D$.  We perform the change of variables $h \rightarrow dh, a \rightarrow ea, b \rightarrow gb$, where $d,e$ and $g$ divide $D^\infty$, and $h,a$ and $b$ are now coprime with $D$. We then obtain
\begin{align*}
&\sum_{\substack{d,e,g \mid D^\infty \\ (e,g)=1}} \frac{\rho(de)\rho(dg)}{de^{1+u}g^{1+v}} \sum_{\substack{h\leq X/d \\ (h,D)=1}} \frac{1}{h} \sum_{\substack{a \leq X/deh \\ b \leq X/dgh \\ (a,b)=1 \\ (ab,D)=1}} \frac{\rho(ha)\rho(hb)f_a(1+u+v)f_b(1+u+v)}{a^{1+u}b^{1+v}},
\end{align*}
since $f_e(1+u+v) =f_g(1+u+v) = 1$ for $e,g|D^\infty$. 

Using the trivial bounds $|\rho(mn)| \leq \tau(mn) \leq \tau(m) \tau(n)$ and $|f_n(1+u+v)| \leq \tau(n)$, we see that the contribution from $e > D$ or $g > D$ is
\begin{align*}
&\ll (\log q)^{12} \prod_{p \mid D} \left(1 + \frac{1}{p}\right)^6\sum_{\substack{e \mid D^\infty \\ e > D}} \frac{\tau(e)}{e} \ll D^{-1/2}(\log q)^{12} \prod_{p \mid D} \left(1 + \frac{1}{p}\right)^6\sum_{e \mid D^\infty} \frac{\tau(e)}{\sqrt{e}} \\
&\ll_\varepsilon D^{-1/2+\varepsilon}(\log q)^{12},
\end{align*}
say, by the divisor bound. Since $D \gg q^\varepsilon$, this bound is acceptably small. The contribution from $d > D$ is similarly bounded by the same error term.

We next use lacunarity to reduce the range of summation of $h$, with an eye towards making the summation ranges of each variable independent of all the others. By the Cauchy-Schwarz inequality and symmetry the contribution from $h >\sqrt{X}$ is
\begin{align*}
\ll_\eps (\log q)^{2+\varepsilon} \sum_{a \leq X} \frac{\tau(a)}{a} \sum_{\sqrt{X} < h \leq X} \frac{\rho(ha)^2}{h}.
\end{align*}
We factor $h \rightarrow wh$, where $w| a^\infty$ and $(h,a) = 1$. The contribution from $w > D$ is $O_\eps(q^{-\varepsilon})$, and the contribution from $w \leq D$ is
\begin{align*}
\ll_\eps (\log q)^{10+\varepsilon} \sum_{\sqrt{X}/ D < h \leq X} \frac{\rho(h)^2}{h} \ll_{\eps,C}L(1,\psi) (\log q)^{15 + \varepsilon}  + (\log q)^{-C}
\end{align*}
by the inequality $|\rho(h)| \leq (1\star\psi)(h)$ and Lemma \ref{lem:divisor twisted by lacunary}.

With $d,e,g$ and $h$ small, we now use lacunarity to make the summation ranges of $a$ and $b$ independent of $d,e,g$ and $h$, so that the sum factors. For instance, the error introduced in replacing the condition $a \leq X/deh$ by $a \leq X$ is bounded by
\begin{align}\label{eqn:use lacunarity to change summation condition for a}
\ll_\eps (\log q)^{8+\varepsilon}\sum_{\substack{\sqrt{X}/D^2 < a \leq X \\ (a,D)=1}} \frac{\tau(a)|f_a(1+u+v)|}{a}.
\end{align}
We factor $a = a_+ a_-$, where $p| a_{\pm}$ implies $\psi(p) = \pm 1$. Since $a_+ a_- > \sqrt{X}/D^2$, one of $a_+$ or $a_-$ is larger than $X^{1/4}/D=D^4$. The quantity in \eqref{eqn:use lacunarity to change summation condition for a} is therefore
\begin{align*}
&\ll_\eps (\log q)^{8+\varepsilon}\sum_{a_+ \leq X} \frac{\tau(a_+) |f_{a_+}(1+u+v)|}{a_+} \sum_{X^{1/4}/D < a_- \leq X} \frac{\tau(a_-) |f_{a_-}(1+u+v)|}{a_-} \\ 
&\qquad\qquad+  (\log q)^{8+\varepsilon} \sum_{a_- \leq X} \frac{\tau(a_-)  |f_{a_-}(1+u+v)|}{a_-}\sum_{X^{1/4}/D < a_+ \leq X} \frac{\tau(a_+) |f_{a_+}(1+u+v)|}{a_+}.
\end{align*}
Since $ f_{a_-}(1+u+v) = \mathbf{1}(a_- = \square)$, this is
\begin{align*}
&\ll_\eps (\log q)^{12+\varepsilon} \sqrt{D} X^{-1/8}+ (\log q)^{8+\varepsilon}\sum_{X^{1/4}/D < a_+ \leq X} \frac{\tau(a_+) |f_{a_+}(1+u+v)|}{a_+}.
\end{align*}
We have $|f_{a_+}(1+u+v)| \leq (1\star\psi)(a_+)$, so the quantity in \eqref{eqn:use lacunarity to change summation condition for a} is
\begin{align*}
\ll_{\eps,C} L(1,\psi)(\log q)^{13+\varepsilon} + (\log q)^{-C}
\end{align*}
by Lemma \ref{lem:divisor twisted by lacunary}. We obtain an identical error term when we change the range of summation of $b$.

Having made the summation conditions independent of one another, we may remove the conditions $d,e,g \leq D$ by the same means we imposed them, and use lacunarity to replace the condition $h \leq \sqrt{X}$ by $h \leq X$. Therefore, up to an error of size $O_{\eps,C}(L(1,\psi)(\log q)^{15+\varepsilon} + (\log q)^{-C})$ we have
\begin{align}\label{eq:main term after deg sum factors}
&\sum_{h \leq X} \frac{1}{h}\sum_{\substack{a,b \leq X/h \\ (a,b)=1}} \frac{\rho(ha) \rho(hb)f_a(1+u+v)f_b(1+u+v)}{a^{1+u}b^{1+v}}\\
&\qquad=\bigg(\sum_{\substack{d,e,g| D^\infty \\ (e,g)=1}} \frac{\rho(de)\rho(dg)}{de^{1+u}g^{1+v}} \bigg) \bigg(\sum_{\substack{h\leq X \\ (h,D)=1}} \frac{1}{h}\sum_{\substack{a,b \leq X \\  (a,b)=1 \\ (ab,D)=1}} \frac{\rho(ha)\rho(hb)f_a(1+u+v)f_b(1+u+v)}{a^{1+u}b^{1+v}} \bigg).\nonumber
\end{align}

We now analyze the factor involving $d,e$, and $g$.

\begin{lemma}\label{lem:euler product mollif lemma}
We have
\begin{align*}
\sum_{d \mid D^\infty} \frac{1}{d}\sum_{\substack{e,g \mid D^\infty \\ (e,g)=1}} \frac{\rho(de)\rho(dg)}{e^{1+u}g^{1+v}} &= \prod_{p \mid D} \bigg(1 + \frac{1}{p}-\frac{1}{p^{1+u}}-\frac{1}{p^{1+v}} \bigg).
\end{align*}
\end{lemma}
\begin{proof}
We observe that since $\rho(p^j)$ vanishes for $p | D$ and $j \geq 2$, we must have $d,e$ and $g$ squarefree. So
\begin{align*}
\sum_{d \mid D^\infty} \frac{1}{d}\sum_{\substack{e,g \mid D^\infty \\ (e,g)=1}} \frac{\rho(de)\rho(dg)}{e^{1+u}g^{1+v}} = \sum_{\substack{e,g \mid D \\ (e,g)=1}} \frac{\mu^2(e)\mu^2(g)}{e^{1+u}g^{1+v}}\sum_{d \mid D} \frac{\mu^2(d)\rho(de)\rho(dg)}{d}.
\end{align*}
We then factor $d \rightarrow e_1g_1d$, where $e_1$ divides $e$, $g_1$ divides $g$ and $d$ is coprime to $eg$. We therefore have
\begin{align*}
\sum_{\substack{e,g \mid D \\ (e,g)=1}} \frac{\mu^2(e)\mu^2(g)}{e^{1+u}g^{1+v}} \sum_{\substack{e_1 \mid e \\ g_1 \mid g}} \frac{\mu^2(e_1) \mu^2(g_1) \rho(e_1)\rho(ee_1)\rho(g_1)\rho(gg_1)}{e_1g_1} \sum_{\substack{d \mid D \\ (d,eg)=1}} \frac{\mu^2(d)\rho(d)^2}{d}.
\end{align*}
We have that $\rho(d)^2 = 1$ for $d| D$ and $d$ squarefree, so the sum over $d$ is
\begin{align*}
\prod_{p \mid e} \bigg(1 + \frac{1}{p} \bigg)^{-1}\prod_{p \mid g} \left(1 + \frac{1}{p} \right)^{-1}\prod_{p \mid D} \left(1 + \frac{1}{p}\right)= h_1(e)h_1(g)\prod_{p \mid D} \left(1 + \frac{1}{p}\right),
\end{align*}
say. Next, observe that the function
\begin{align*}
\sum_{e_1 \mid e}\frac{\mu^2(e_1)\rho(e_1)\rho(ee_1)}{e_1}
\end{align*}
is multiplicative in $e$. Since $e$, $g$ are squarefree divisors of $D$, it suffices to evaluate it on primes dividing $D$. We have
\begin{align*}
\sum_{e_1 \mid p}\frac{\mu^2(e_1)\rho(e_1)\rho(pe_1)}{e_1} = \rho(p) + \frac{\rho(p)\rho(p^2)}{p} = \mu(p),
\end{align*}
since $\psi(p) = 0$ for $p| D$. The original factor involving $d,e,g$ is therefore
\begin{align*}
\prod_{p \mid D} \left(1 + \frac{1}{p}\right)\sum_{e \mid D} \frac{\mu(e)h_1(e)}{e^{1+u}}\sum_{\substack{g \mid D \\ (g,e)=1}} \frac{\mu(g)h_1(g)}{g^{1+v}}.
\end{align*}
The sum over $g$ is
\begin{align*}
 \prod_{\substack{p \mid D \\ p \nmid e}} \left(1 - \frac{h_1(p)}{p^{1+v}} \right) =  \prod_{p| e} \left(1 - \frac{h_1(p)}{p^{1+v} }\right)^{-1} \prod_{p| D} \left(1 - \frac{h_1(p)}{p^{1+v}} \right).
\end{align*}
As
\begin{align*}
&\prod_{p| D} \left(1 - \frac{h_1(p)}{p^{1+v}} \right)\sum_{e \mid D} \frac{\mu(e)h_1(e)}{e^{1+u}}\prod_{p|e} \left(1 - \frac{h_1(p)}{p^{1+v}} \right)^{-1} = \prod_{p \mid D}\left(1 - \frac{h_1(p)}{p^{1+u}}- \frac{h_1(p)}{p^{1+v}}\right),
\end{align*}
the claim follows.
\end{proof}

By \eqref{fID}, \eqref{eq:main term after deg sum factors} and Lemma \ref{lem:euler product mollif lemma}, we have
\begin{align}\label{101}
J(u,v)& =\mathcal{A}(u,v)K(u,v)+ O_\eps \big(L(1,\psi) (\log q)^{19+\varepsilon}\big) +O_C\big( (\log q)^{-C} \big),
\end{align}
where
\begin{equation}\label{mathcalA}
\mathcal{A}(u,v) =\prod_{p \mid D} \bigg(1 -\frac{1}{p^{1+u+v}} \bigg)^{-1}\bigg(1 + \frac{1}{p}-\frac{1}{p^{1+u}}-\frac{1}{p^{1+v}} \bigg)
\end{equation}
and
\begin{align}\label{102}
K(u,v)&=\prod_{\psi(p)=-1} \bigg(1 -\frac{1}{p^{2(1+u+v)}} \bigg)^{-1}\prod_{\psi(p)=1} \bigg(1 -\frac{1}{p^{2(1+u+v)}} \bigg)\\
&\qquad\qquad \sum_{\substack{n\geq1\\p|n\rightarrow \psi(p)=1}}\frac{\tau_4(n)}{n^{1+u+v}}\sum_{\substack{h\leq X \\ (h,D)=1}} \frac{1}{h}\sum_{\substack{a,b \leq X \\  (a,b)=1 \\ (ab,D)=1}} \frac{\rho(ha)\rho(hb)f_a(1+u+v)f_b(1+u+v)}{a^{1+u}b^{1+v}}.\nonumber
\end{align}

We now turn to primes $p$ with $\psi(p) = -1$. We factor $h=h_+h_-$, $a=a_+a_-$ and $b=b_+b_-$, where $p|h_+a_+b_+$ implies that $\psi(p) = 1$, and $p|h_-a_-b_-$ implies that $\psi(p) =- 1$. Using lacunarity to make summation conditions independent of variables as before, we find that up to an error of size $O_{\eps,C}(L(1,\psi)(\log q)^{15+\varepsilon} + (\log q)^{-C})$,
\begin{align}\label{103}
&\sum_{\substack{h\leq X \\ (h,D)=1}} \frac{1}{h}\sum_{\substack{a,b \leq X \\  (a,b)=1 \\ (ab,D)=1}} \frac{\rho(ha)\rho(hb)f_a(1+u+v)f_b(1+u+v)}{a^{1+u}b^{1+v}}\nonumber\\
&\qquad\qquad= \bigg(\sum_{h_-\leq X} \frac{1}{h_-} \sum_{\substack{a_-,b_- \leq X \\  (a_-,b_-)=1}} \frac{\rho(h_-a_-)\rho(h_-b_-)f_{a_-}(1+u+v)f_{b_-}(1+u+v)}{a_-^{1+u}b_-^{1+v}} \bigg) \\ 
&\qquad\qquad\qquad\qquad\bigg(\sum_{h_+\leq X} \frac{1}{h_+} \sum_{\substack{a_+,b_+ \leq X \\  (a_+,b_+)=1}} \frac{\rho(h_+a_+)\rho(h_+b_+)f_{a_+}(1+u+v)f_{b_+}(1+u+v)}{a_+^{1+u}b_+^{1+v}} \bigg).\nonumber
\end{align}
Recalling that $f_{a_-}(1+u+v) = \mathbf{1}(a_- = \square)$, we have
\begin{align*}
&\sum_{h_-\leq X} \frac{1}{h_-} \sum_{\substack{a_-,b_- \leq X \\  (a_-,b_-)=1}} \frac{\rho(h_-a_-)\rho(h_-b_-)f_{a_-}(1+u+v)f_{b_-}(1+u+v)}{a_-^{1+u}b_-^{1+v}}\\
&\qquad\qquad= \sum_{h_-\leq X} \frac{1}{h_-}\sum_{(a_-,b_-)=1} \frac{\rho(h_-a_-^2)\rho(h_-b_-^2)}{a_-^{2(1+u)}b_-^{2(1+v)}} + O_\eps(q^{-\varepsilon}).
\end{align*}
Since $\rho$ is supported on cube-free integers, we see that $(h_-,a_-b_-) = 1$. So the above expression is
\begin{align*}
\sum_{h_-} \frac{\mu^2(h_-)}{h_-^2} \sum_{\substack{(a_-,b_-)=1 \\ (a_-b_-,h_-)=1}} \frac{\mu(a_-) \mu(b_-)}{a_-^{2(1+u)}b_-^{2(1+v)}} + O_\eps(q^{-\varepsilon}),
\end{align*}
since $\rho(p) = 0$ for $\psi(p) = -1$, and $\rho(p^2) = \psi(p) = -1 = \mu(p)$.

We rearrange the order of summation to have the sum on $h_-$ as the innermost sum. Using the Euler products, we find
\begin{align}\label{defh2}
\sum_{(h_-,a_-b_-)=1} \frac{\mu^2(h_-)}{h_-^2} &= \prod_{p \mid a_-b_-}\bigg(1 + \frac{1}{p^2} \bigg)^{-1} \prod_{\psi(p) = -1} \bigg(1 + \frac{1}{p^2} \bigg)= h_2(a_-)h_2(b_-) \prod_{\psi(p) = -1} \left(1 + \frac{1}{p^2} \right),
\end{align}
say (we have used the fact that $a_-$ and $b_-$ are coprime). Now
\begin{align*}
\sum_{(b_-,a_-)=1} \frac{\mu(b_-)h_2(b_-)}{b_-^{2(1+v)}} = \prod_{p|a_-} \bigg(1-\frac{1}{p^{2v}(p^2+1)}\bigg)^{-1} \prod_{\psi(p) = -1} \bigg(1-\frac{1}{p^{2v}(p^2+1)}\bigg),
\end{align*}
and the sum over $a_-$ becomes
\begin{align*}
\prod_{\psi(p) = -1} \left(1 + \frac{1}{p^2}-\frac{1}{p^{2(1+u)}}-\frac{1}{p^{2(1+v)}} \right).
\end{align*}
Hence it follows from \eqref{101}, \eqref{102} and \eqref{103} that
\begin{align}\label{502}
J(u,v) = \mathcal{A}(u,v)\mathcal{B}(u,v)L(u,v)+ O_\eps \big(L(1,\psi) (\log q)^{19+\varepsilon}\big) +O_C\big( (\log q)^{-C} \big),
\end{align}
where
\begin{equation}\label{mathcalB}
\mathcal{B}(u,v) =\prod_{\psi(p) = -1} \bigg(1 -\frac{1}{p^{2(1+u+v)}} \bigg)^{-1}\left(1 + \frac{1}{p^2}-\frac{1}{p^{2(1+u)}}-\frac{1}{p^{2(1+v)}} \right)
\end{equation}
and
\begin{align*}
&L(u,v)=\prod_{\psi(p) =1} \bigg(1 -\frac{1}{p^{2(1+u+v)}} \bigg)\nonumber \\
&\qquad\qquad \sum_{\substack{n\geq1\\p|n\rightarrow \psi(p)=1}}\frac{\tau_4(n)}{n^{1+u+v}}\sum_{h_+\leq X} \frac{1}{h_+} \sum_{\substack{a_+,b_+ \leq X \\  (a_+,b_+)=1}} \frac{\rho(h_+a_+)\rho(h_+b_+)f_{a_+}(1+u+v)f_{b_+}(1+u+v)}{a_+^{1+u}b_+^{1+v}}.
\end{align*}

For the rest of the subsection we assume that all integer variables are supported on integers comprised only of primes with $\psi(p) = 1$, and we do not indicate this in the notation anymore. It is helpful to remember the inequality
\begin{align*}
\mathbf{1}(p| n \Rightarrow \psi(p) = 1) \leq (1\star\psi)(n).
\end{align*}

Since $\rho$ is supported on cube-free numbers, we see that $h$ is cube-free. We factor $h = \alpha\beta^2$, where $\alpha,\beta$ are squarefree and coprime. Thus the sum over $a,b,h$ becomes
\begin{align*}
\sum_{\substack{\alpha\beta^2 \leq X \\ (\alpha,\beta)=1}} \frac{\mu^2(\alpha)\mu^2(\beta)}{\alpha\beta^2}  \sum_{\substack{a,b \leq X \\ (a,b)=1}} \frac{\rho(\alpha\beta^2 a) \rho(\alpha\beta^2 b)f_a(1+u+v)f_b(1+u+v)}{a^{1+u}b^{1+v}}.
\end{align*}
The support of $\rho$ implies we may take $\beta$ to be coprime to $a$ and $b$, so our sum becomes
\begin{align*}
\sum_{\alpha \leq X} \frac{\mu^2(\alpha)}{\alpha} \sum_{\substack{a,b \leq X \\ (a,b)=1}} \frac{\rho(\alpha a)\rho(\alpha b)f_a(1+u+v)f_b(1+u+v)}{a^{1+u}b^{1+v}} \sum_{\substack{\beta \leq \sqrt{X/\alpha} \\ (\beta,\alpha ab)=1}} \frac{\mu^2(\beta) \rho(\beta^2)^2}{\beta^2}.
\end{align*}
Observe that $\rho(\beta^2)^2 = 1$ for squarefree $\beta$. We now factor $a \rightarrow da, b \rightarrow eb$, where $d,e| \alpha^\infty$ and $a,b$ are coprime to $\alpha$. We then use trivial estimates and Lemma \ref{lem:divisor twisted by lacunary} to make summations independent from one another. Up to an error of size $O_{\eps,C}(L(1,\psi)(\log q)^{13+\varepsilon} + (\log q)^{-C})$, we have
\begin{align*}
&\sum_{\alpha \leq X} \frac{\mu^2(\alpha)}{\alpha} \sum_{\substack{d,e \mid \alpha^\infty \\ (d,e)=1}}\frac{\rho(\alpha d) \rho(\alpha e) f_d(1+u+v)f_e(1+u+v)}{d^{1+u}e^{1+v}} \\ 
&\qquad\qquad \sum_{\substack{a,b \leq X \\ (a,b)=1 \\ (ab,\alpha)=1}} \frac{\rho(a)\rho(b)f_a(1+u+v)f_b(1+u+v)}{a^{1+u}b^{1+v}} \sum_{\substack{\beta \leq \sqrt{X/\alpha} \\ (\beta,\alpha ab)=1}} \frac{\mu^2(\beta)}{\beta^2}.
\end{align*}

By the support of $\rho$, we may take $d$ and $e$ to be squarefree. Also, we see that the function
\begin{align*}
h_3 (\alpha;u,v) := \sum_{\substack{d,e \mid \alpha \\ (d,e)=1}}\frac{\mu^2(d) \mu^2(e)\rho(\alpha d) \rho(\alpha e) f_d(1+u+v)f_e(1+u+v)}{d^{1+u}e^{1+v}}
\end{align*}
is multiplicative in $\alpha$. Since $\alpha$ is squarefree, it suffices to evaluate $h_3$ on primes, and we find
\begin{align*}
h_3(p;u,v) &= \rho(p)^2 + \frac{\rho(p)\rho(p^2)f_p(1+u+v)}{p^{1+u}} + \frac{\rho(p)\rho(p^2)f_p(1+u+v)}{p^{1+v}}\\
&=4-\frac{4}{p}\bigg(1+\frac{1}{p^{1+u+v}}\bigg)^{-1}\bigg(\frac{1}{p^u}+\frac{1}{p^v}\bigg).
\end{align*}
We therefore have
\begin{align*}
&\sum_{\alpha \leq X} \frac{\mu^2(\alpha)h_3(\alpha;u,v)}{\alpha}\sum_{\substack{a,b \leq X \\ (a,b)=1 \\ (ab,\alpha)=1}} \frac{\rho(a)\rho(b)f_a(1+u+v)f_b(1+u+v)}{a^{1+u}b^{1+v}} \sum_{(\beta,\alpha ab)=1} \frac{\mu^2(\beta)}{\beta^2}\\
&\qquad\qquad + O_\eps\big(L(1,\psi)(\log q)^{17+\varepsilon}\big)+O_C\big((\log q)^{-C}\big),
\end{align*}
where we have used Lemma \ref{lem:divisor twisted by lacunary} to reduce the range of $\alpha$ to $a \leq X^{1/5}$, say, then extended the sum on $\beta$ to infinity, and then increased the range of $\alpha$ by Lemma \ref{lem:divisor twisted by lacunary} again. As
\begin{align*}
\sum_{(\beta,\alpha ab)=1} \frac{\mu^2(\beta)}{\beta^2} = \prod_{p \nmid \alpha ab} \bigg(1 + \frac{1}{p^2}\bigg) = h_2(\alpha) h_2(a) h_2(b) \prod_p \left(1 + \frac{1}{p^2}\right),
\end{align*}
where $h_2$ is the multiplicative function which was defined in \eqref{defh2}, we obtain
\begin{align*}
&\prod_p \left(1 + \frac{1}{p^2}\right)\sum_{\alpha \leq X} \frac{\mu^2(\alpha)h_2(\alpha)h_3(\alpha;u,v)}{\alpha}\sum_{\substack{a,b \leq X \\ (a,b)=1 \\ (ab,\alpha)=1}} \frac{\rho(a)\rho(b)f_a(1+u+v)f_b(1+u+v)h_2(a) h_2(b)}{a^{1+u}b^{1+v}}\\
&\qquad\qquad + O_\eps\big(L(1,\psi)(\log q)^{17+\varepsilon}\big)+O_C\big((\log q)^{-C}\big).
\end{align*}

We use Lemma \ref{lem:divisor twisted by lacunary} to deduce
\begin{align*}
&\sum_{\substack{a,b\leq X \\ (a,b)=1 \\ (ab,\alpha)=1}}\frac{\rho(a)\rho(b)f_a(1+u+v)f_b(1+u+v)h_2(a) h_2(b)}{a^{1+u}b^{1+v}}\\
&\qquad\qquad= \sum_{\substack{m \leq X \\ (m,\alpha)=1}} \frac{g_1(m;u,v)}{m} + O_\eps\big(L(1,\psi)(\log q)^{17+\varepsilon}\big)+O_C\big((\log q)^{-C}\big),
\end{align*}
where $g_1(m;u,v)$ is the multiplicative function in $m$ given by
\begin{align*}
g_1(p;u,v) &=\frac{\rho(p)f_p(1+u+v)h_2(p)}{p^u}+\frac{\rho(p)f_p(1+u+v)h_2(p)}{p^v}\\
&=-4h_2(p) \bigg(1+\frac{1}{p^{1+u+v}}\bigg)^{-1}\bigg(\frac{1}{p^{u}}+\frac{1}{p^{v}}\bigg), \\
g_1(p^2;u,v) &=\frac{\rho(p^2)f_{p^2}(1+u+v)h_2(p^2)}{p^{2u}}+\frac{\rho(p^2)f_{p^2}(1+u+v)h_2(p^2)}{p^{2v}}\\
&=h_2(p) \bigg(1+\frac{1}{p^{1+u+v}}\bigg)^{-1}\bigg(3-\frac{1}{p^{1+u+v}}\bigg)\bigg(\frac{1}{p^{2u}}+\frac{1}{p^{2v}}\bigg), \\
g_1(p^j;u,v) &= 0 \ \ \ \ \ \ \ \ \ \text{if }j \geq 3.
\end{align*}
So the sum over $a,b,h$ is
\begin{align*}
&\prod_p \left(1 + \frac{1}{p^2}\right)\sum_{\substack{\alpha,m \leq X \\ (\alpha,m)=1}} \frac{\mu^2(\alpha)h_2(\alpha)h_3(\alpha;u,v)g_1(m;u,v)}{\alpha m}\\
&\qquad\qquad+ O_\eps \big(L(1,\psi)(\log q)^{21+\varepsilon}\big) +O_C\big( q (\log q)^{-C} \big).
\end{align*}

Applying Lemma \ref{lem:divisor twisted by lacunary} again we obtain
\begin{align*}
&\sum_{\substack{\alpha,m \leq X \\ (\alpha,m)=1}} \frac{\mu^2(\alpha)h_2(\alpha)h_3(\alpha;u,v)g_1(m;u,v)}{\alpha m}\\
&\qquad\qquad= \sum_{m \leq X} \frac{g_2(m;u,v)}{m}+ O_\eps\big(L(1,\psi)(\log q)^{21+\varepsilon}\big)+O_C\big((\log q)^{-C}\big),
\end{align*}
where $g_2(m;u,v)$ is the multiplicative function in $m$ given by
\begin{align*}
g_2(p;u,v) &= h_2(p)h_3(p;u,v) + g_1(p) =4h_2(p)-4h_2(p) \bigg(1+\frac{1}{p^{1+u+v}}\bigg)^{-1}\bigg(1+\frac1p\bigg)\bigg(\frac{1}{p^{u}}+\frac{1}{p^{v}}\bigg), \\
g_2(p^2;u,v) &= g_1(p^2) =h_2(p) \bigg(1+\frac{1}{p^{1+u+v}}\bigg)^{-1}\bigg(3-\frac{1}{p^{1+u+v}}\bigg)\bigg(\frac{1}{p^{2u}}+\frac{1}{p^{2v}}\bigg), \\
g_2(p^j;u,v) &= 0 \ \ \ \ \ \ \ \ \ \ \ \text{if }j \geq 3.
\end{align*}
Hence
\begin{align*}
L(u,v)&=\prod_{p} \bigg(1 -\frac{1}{p^{2(1+u+v)}} \bigg)\left(1 + \frac{1}{p^2}\right) \sum_{\substack{m \leq X\\ n\geq1}} \frac{g_2(m;u,v)\tau_4(n)}{mn^{1+u+v}}\\
&\qquad\qquad+ O_\eps \big(L(1,\psi)(\log q)^{25+\varepsilon}\big) +O_C\big( q (\log q)^{-C} \big),
\end{align*}

We wish to truncate the sum over $n$ to $n \leq X$ in order to facilitate moving the contours of integration to the left. We perform this truncation in two steps. First, we see that the contribution from $n > Y := \exp((\log D)^3)$ is
\begin{align*}
\left|\sum_{n > Y}\frac{\tau_4(n)}{n^{1+u+v}}\right| \ll Y^{-c} \sum_{n\geq1} \frac{\tau_4(n)}{n^{1+c}} = Y^{-c} \zeta(1+c)^4 \ll Y^{-1/\log q} (\log q)^4 \ll_C D^{-C}.
\end{align*}
Recalling that $n$ is supported on primes with $\psi(p) = 1$, we have that the contribution from $X < n \leq Y$ is
\begin{align*}
&\ll \sum_{X < n \leq Y} \frac{\tau_4(n)}{n} \leq \sum_{X < ab \leq Y} \frac{(1\star \psi)(a)(1\star\psi)(b)}{ab} \\ 
&\ll \sum_{a \leq Y}\frac{(1\star\psi)(a)}{a}\sum_{\sqrt{X} < b \leq Y} \frac{(1\star \psi)(b)}{b} \ll L(1,\psi) (\log D)^9,
\end{align*}
the last inequality following by trivial estimation and \eqref{eq: lacunarity}. By another application of Lemma \ref{lem:divisor twisted by lacunary} we then find that
\begin{align}\label{553}
L(u,v)&=\mathcal{C}(u,v)+ O_\eps \big(L(1,\psi)q (\log q)^{25+\varepsilon}\big) +O_C\big( q (\log q)^{-C} \big),
\end{align}
where
\begin{align}\label{mathcalC}
\mathcal{C}(u,v)&=\prod_{\psi(p)=1} \bigg(1 -\frac{1}{p^{2(1+u+v)}} \bigg)\left(1 + \frac{1}{p^2}\right) \sum_{\substack{n \leq X\\p|n\rightarrow \psi(p)=1}} \frac{g_3(n;u,v)}{n}
\end{align}
and $g_3(n;u,v)$ is the multiplicative function in $n$ given by
\begin{align*}
g_3(p;u,v) &= g_2(p;u,v)+\frac{4}{p^{u+v}}, \\
g_3(p^j;u,v) &= \frac{\tau_4(p^j)}{p^{j(u+v)}}+g_2(p;u,v)\frac{\tau_4(p^{j-1})}{p^{(j-1)(u+v)}}+g_2(p^2;u,v)\frac{\tau_4(p^{j-2})}{p^{(j-2)(u+v)}} \ \ \ \ \ \ \ \ \ \ \ \text{if }j \geq 2.
\end{align*}
Thus, from \eqref{550}, \eqref{502} and \eqref{553},
\begin{align}\label{554}
I^D& =\frac{\phi(q)}{2}\frac{1}{(2\pi i)^2}\int_{(c)}\int_{(c)}\widetilde{V}(u)\widetilde{V}(v)Q^{u+v}\mathcal{A}(u,v)\mathcal{B}(u,v)\mathcal{C}(u,v)dudvdw\\
&\qquad\qquad -\frac{\phi(q)}{2}V\Big(\frac{1}{Q}\Big)^2+ O_\eps \big(L(1,\psi)q (\log q)^{25+\varepsilon}\big) +O_C\big( q (\log q)^{-C} \big),\nonumber
\end{align}
where $\mathcal{A}(u,v)$, $\mathcal{B}(u,v)$ and $\mathcal{C}(u,v)$ are given in \eqref{mathcalA}, \eqref{mathcalB} and \eqref{mathcalC}, respectively.

Observe that
\[
\mathcal{A}(u,0)=\mathcal{A}(0,v)=\mathcal{B}(u,0)=\mathcal{B}(0,v)=1.
\]Also, note that $\mathcal{B}(u,v)$ and $\mathcal{C}(u,v)$ are absolutely convergent provided that
\[
\text{Re}(u),\text{Re}(v)>-1/4+\eps,
\]
say. Uniformly in this region we have
\[
\mathcal{A}(u,v),\,\partial_v\mathcal{A}(u,0)\ll_\eps D^\eps,\qquad\partial_v\mathcal{A}(0,0),\,\partial^2_{uv}\mathcal{A}(0,0)\ll_\eps (\log D)^{\eps},
\]
\[
\mathcal{B}(u,v),\,\partial_v\mathcal{B}(u,v),\,\partial^2_{uv}\mathcal{B}(u,v)\ll_\eps 1
\]
and, as $g_3(n;u,v)\ll_\eps n^{-\text{Re}(u+v)+\eps}$,
\[
\mathcal{C}(u,v),\,\partial_v\mathcal{C}(u,v),\,\partial^2_{uv}\mathcal{C}(u,v)\ll_\eps X^{-\text{Re}(u+v)+\eps}.
\]

\subsection{Wrapping up}

Consider the double integrals in \eqref{554}. 
We first move the $v$-contour to  $\text{Re}(v) = -1/4+\varepsilon$, crossing a double pole at $v=0$ with residue
\begin{align*}
R_1&=\frac{v_1}{2\pi i}\int_{(c)}\widetilde{V}(u)Q^u\mathcal{C}(u,0)du\\
&\qquad\qquad+\frac{v_2}{2\pi i(\log Q)}\int_{(c)}\widetilde{V}(u)Q^u\Big(\mathcal{C}(u,0)\big(\log Q+\partial_v\mathcal{A}(u,0)+\partial_v\mathcal{B}(u,0)\big)+\partial_v\mathcal{C}(u,0)\Big)du.
\end{align*}
For the new integral, we move the $u$-contour to  $\text{Re}(u) = -1/4+\eps$, encountering a double pole at $u=0$. The residue is 
\begin{align*}
R_2&=\frac{v_1}{2\pi i}\int_{(-1/4+\eps)}\widetilde{V}(v)Q^v\mathcal{C}(0,v)dv\\
&\quad\quad+\frac{v_2}{2\pi i(\log Q)}\int_{(-1/4+\eps)}\widetilde{V}(v)Q^v\Big(\mathcal{C}(0,v)\big(\log Q+\partial_u\mathcal{A}(0,v)+\partial_u\mathcal{B}(0,v)\big)+\partial_u\mathcal{C}(0,v)\Big)dv
\end{align*}
and the final integral is bounded trivially by
\[
\ll_\eps Q^{-1/2+\eps}X^{1/2}\ll_\eps q^{-1/2+\eps}D^{-1/4}X^{1/2}.
\]
We can also bound $R_2$ trivially by $O_\eps(Q^{-1/4+\eps}X^{1/4})$. 

Regarding $R_1$, we move the line of integration to $\text{Re}(u) = -1/4+\eps$, crossing a double pole at $u=0$ with residue
\begin{align*}
(v_1+v_2)^2\mathcal{C}(0,0)+\frac{(v_1+v_2)v_2}{\log Q}\big(\partial_u\mathcal{C}(0,0)+\partial_v\mathcal{C}(0,0)\big)+\frac{v_2^2}{(\log Q)^2}\partial^2_{uv}\mathcal{C}(0,0)+O_\eps\big((\log Q)^{-1+\eps}\big),
\end{align*}
and the new integral, like $R_2$, is bounded trivially by $O_\eps(q^{-1/4+\eps}D^{-1/8}X^{1/4})$.

Next consider $\mathcal{C}(0,0)$. It is easy to verify that $g_3(n;0,0)\ll 1/n$, so we can extend the sum over $n$ in \eqref{mathcalC} to infinity and use multiplicativity to obtain
\begin{align*}
\mathcal{C}(0,0)&=\prod_{\psi(p)=1} \bigg(1 -\frac{1}{p^{2}} \bigg)\bigg(1 +\frac{1}{p^{2}} \bigg)\bigg[1 + \frac{g_2(p;0,0)}{p}+\frac4p\\
&\qquad\qquad+\sum_{j\geq2}\bigg( \frac{\tau_4(p^j)}{p^{j}}+g_2(p;0,0)\frac{\tau_4(p^{j-1})}{p^{j}}+g_2(p^2;0,0)\frac{\tau_4(p^{j-2})}{p^{j}}\bigg)\bigg]+O_\varepsilon\big(X^{-1+\varepsilon}\big)\\
&=\prod_{\psi(p)=1} \bigg(1 -\frac{1}{p^{2}} \bigg)h_2(p)^{-1}\bigg(1+\sum_{j\geq1}\frac{\tau_4(p^j)}{p^{j}}\bigg)\bigg(1+\frac{g_2(p;0,0)}{p}+\frac{g_2(p^2;0,0)}{p^2}\bigg)+O_\varepsilon\big(X^{-1+\varepsilon}\big)\\
&=\prod_{\psi(p)=1} \bigg(1 +\frac{1}{p} \bigg)\bigg(1-\frac{1}{p}\bigg)^{-3}\bigg[h_2(p)^{-1}-\frac{4}{p}+\frac{2}{p^2}\bigg(1+\frac{1}{p}\bigg)^{-1}\bigg(3-\frac{1}{p}\bigg)\bigg]+O_\varepsilon\big(X^{-1+\varepsilon}\big)\\
&=\prod_{\psi(p)=1}\bigg(1-\frac{1}{p}\bigg)^{-3}\bigg[\bigg(1-\frac{4}{p}+\frac{1}{p^2}\bigg) \bigg(1 +\frac{1}{p} \bigg)+\frac{2}{p^2}\bigg(3-\frac{1}{p}\bigg)\bigg]+O_\varepsilon\big(X^{-1+\varepsilon}\big)\\
&=1+O_\varepsilon\big(X^{-1+\varepsilon}\big).
\end{align*}
Similarly, one can verify that $\partial_vg_3(n;0,0),\,\partial^2_{uv}g_3(n;0,0)\ll_\eps n^{-1+\eps}$, and standard calculations then imply that 
\[
\partial_u\mathcal{C}(0,0),\,\partial_v\mathcal{C}(0,0),\,\partial^2_{uv}\mathcal{C}(0,0)\ll 1.
\]
Hence the double integrals in \eqref{554} is
\[
(v_1+v_2)^2+O_\eps\big((\log Q)^{-1+\eps}\big)+O_\eps\big(q^{-1/4+\eps}D^{-1/8}X^{1/4}\big).
\]

Since
\[
V\Big(\frac{1}{Q}\Big)=\frac{1}{2\pi i}\int_{(1)}\widetilde{V}(u)Q^udu=v_1+v_2+O_\eps\big(Q^{-1/2+\eps}\big),
\]
it follows from \eqref{554} that
\begin{align}\label{boundID}
I_D\ll_\eps L(1,\psi)q (\log q)^{25+\varepsilon}+q(\log q)^{-1+\eps}+O_\eps\big(q^{3/4+\eps}D^{-1/8}X^{1/4}\big).
\end{align}

\section{Shifted convolution sum}

In this section we focus on the sum
\[
\sum_{\substack{am\equiv \pm bn(\text{mod}\ q)\\am\ne bn}}(1\star\psi)(m)(1\star\psi)(n)
\]
over dyadic intervals. 

\begin{proposition}\label{propoff}
Let $a,b\in\mathbb{N}$ be coprime, $D\nmid a$ and $D\nmid b$, and let $M,N\geq1$. Let $\omega_1$ and $\omega_2$ be smooth functions supported in $[1,2]$ such that $\omega_1^{(j)}, \omega_2^{(j)}\ll_j q^{\eps}$ for any fixed $j\geq0$. Let
\begin{equation*}
\mathcal S_{a,b} =\sum_{\substack{am\equiv \pm bn(\emph{mod}\ q)\\am\ne bn}}(1\star\psi)(m)(1\star\psi)(n)\omega_1\Big(\frac{m}{M}\Big)\omega_2\Big(\frac{n}{N}\Big).
\end{equation*}
Then
\begin{align*}
\mathcal S_{a,b}&=\frac{L(1,\psi)^2}{ab}\sum_{r\ne 0}\mathfrak{S}_{a,b}(r)\int\omega_1\Big(\frac{x}{aM}\Big)\omega_2\Big(\frac{\mp(qr-x)}{bN}\Big) dx\\
&\qquad\qquad+O_\eps\big( (qMN)^\eps q^{-1}D(abMN)^{1/4}(aM+bN)^{5/4}\big) ,
\end{align*}
where
\begin{align}\label{mathfrakS}
\mathfrak{S}_{a,b}(r)&=\sum_{\ell\geq1}\frac{\big(\psi(\ell_a\ell_b)+\mathbf{1}_{D|(\ell_a,\ell_b)}D\psi(a'b')\big)S(r,0;\ell)}{\ell_a\ell_b}.
\end{align}
Here $a ' =a/(a, \ell)$, $\ell_a = \ell/(a, \ell)$, $b ' = b/(b, \ell)$ and $\ell_b = \ell/(b,\ell)$.

\end{proposition}

\begin{proof}
We write $am\equiv \pm bn(\text{mod}\ q)$ and $am\ne bn$ as $am\mp bn=q r$ with $$0< |r|\leq R=2(aM+bN)q^{-1}.$$ Then
\begin{equation}\label{555}
\mathcal{S}_{a,b}=\sum_{0<|r|\leq R}\sum_{am\mp bn=qr}(1\star\psi)(m)(1\star\psi)(n)\omega_1\Big(\frac{m}{M}\Big)\omega_2\Big(\frac{n}{N}\Big).
\end{equation}

We use the delta method, as developed by Duke, Friedlander and Iwaniec in [\textbf{\ref{DFI}}]. 
As  usual, let
$\delta(0)=1$ and $\delta(n)=0$ for $n\ne 0$. Let $Z\leq \min\{\sqrt{aM},\sqrt{bN}\}$. Then
\begin{equation}\label{deltafnc}
\delta(n) = \sum_{\ell\geq 1}\, \sideset{}{^*}\sum_{k(\text{mod}\ \ell)} e\Big(\frac{kn}{\ell}\Big) \Delta_\ell(n),
\end{equation}
where $\Delta_\ell(u)$ is some smooth function that vanishes if $|u| \leq U=Z^2$ and $\ell \geq 2Z$ (see [\textbf{\ref{DFI}}; Section 4]), and satisfies (see [\textbf{\ref{DFI}}; Lemma 2])
\begin{equation}\label{bdDelta}
\Delta_\ell(u)\ll \big(\ell Z+Z^2\big)^{-1}+\big(\ell Z+|u|\big)^{-1}.
\end{equation}
It is also convenient to attach to both sides of \eqref{deltafnc} a redundant factor $\varphi(n)$, where $\varphi(u)$ is a smooth function supported on $|u|<U$ satisfying $\varphi(0)=1$ and $\varphi^{(j)}(u)\ll U^{-j}$ for any fixed $j\geq0$. Applying this to \eqref{555} we see that
\begin{align}\label{summn}
\mathcal S_{a,b} =\sum_{0<|r|\leq R}\sum_{\ell<2Z}\, \sideset{}{^*}\sum_{k(\text{mod}\ \ell)}e\Big(\frac{-qrk}{\ell}\Big)\sum_{m,n}(1\star\psi)(m)(1\star\psi)(n)e\Big(\frac{k(a m \mp bn)}{\ell}\Big)g_{a,b}(m,n),
\end{align}
where
\[
g_{a,b}(m,n)=\Delta_\ell(am\mp bn-qr)\varphi(am\mp bn-qr)\omega_1\Big(\frac{m}{M}\Big)\omega_2\Big(\frac{n}{N}\Big).
\]

Let $a ' =a/(a, \ell)$, $\ell_a = \ell/(a, \ell)$, $b ' = b/(b, \ell)$ and $\ell_b = \ell/(b,\ell)$.  Further, let $D_{1, a} = (\ell_a, D)$, $D_{2, a} = D/D_{1,a}$ and similarly define $D_{1, b}$, $D_{2, b}$.  By the Chinese Remainder Theorem, we have that $\psi = \psi_{D_{1,a}}\psi_{D_{2, a}}=\psi_{D_{1,b}}\psi_{D_{2, b}}$ for unique $\psi_{D_{i, a}}$ and $\psi_{D_{i, b}}$ characters modulo $D_{i, a}$ and $D_{i, b}$, respectively. We apply the Voronoi summation
formula stated in Theorem \ref{Voronoithm} to the sums over $m,n$ in \eqref{summn}. In doing so, we can write $\mathcal{S}_{a,b}$ as a
principal term plus eight error terms,
\begin{equation}\label{afterVoronoi}
\mathcal{S}_{a,b}=\mathcal{M}_{a,b}+\sum_{i=1}^{8}\mathcal{E}_{i;a,b}.
\end{equation}

We first deal with the error terms. All the eight error terms can be treated similarly, so we only focus here on one of them, 
\begin{align*}
\mathcal{E}_{1;a,b}&=\tau(\psi_{D_{2,a}})\tau(\psi_{D_{2,b}})\sum_{0<|r|\leq R}\sum_{\ell<2Z}\, \sideset{}{^*}\sum_{k(\text{mod}\ \ell)}e\Big(\frac{-qrk}{\ell}\Big)\frac{\psi_{D_{1,a}}(-ka')\psi_{D_{2,a}}(\ell_a)\psi_{D_{1,b}}(-kb')\psi_{D_{2,b}}(\ell_b)}{D_{2,a}D_{2,b}}\\
&\qquad\qquad \sum_{m',n'}\big(\psi_{D_{1,a}}*\psi_{D_{2, a}}\big)(m')\big(\psi_{D_{1,b}}*\psi_{D_{2, b}}\big)(n')e\Big(\frac{\overline{ka'D_{2,a}}m'}{\ell_a}\mp\frac{\overline{kb'D_{2,b}}n'}{\ell_b}\Big)G_{a,b}(m',n'),
\end{align*}
where
\[
G_{a,b}(m',n')=\frac{4\pi^2}{\ell_a\ell_b}\int_{0}^{\infty}\int_{0}^{\infty} g_{a,b}(x,y)Y_0\Big(\frac{4\pi\sqrt{m'x}}{\ell_a\sqrt{D_{2,a}}}\Big)Y_0\Big(\frac{4\pi\sqrt{n'y}}{\ell_b\sqrt{D_{2,b}}}\Big)dxdy.
\]
Here $Y_0(x)$ is a Bessel function of the second kind. As $Z\leq \min\{\sqrt{aM},\sqrt{bN}\}$ we have 
\[
g_{a,b}^{(ij)}\ll_{i,j} \frac{1}{\ell Z}\Big(\frac{a}{\ell Z}\Big)^{i}\Big(\frac{b}{\ell Z}\Big)^{j}
\]
for any fixed $i,j\geq0$. Using the recurrence formula $(x^\nu Y_\nu(x))'=x^\nu Y_{\nu-1}(x)$ and the bound $Y_\nu(x)\ll x^{-1/2}$, integration by parts then implies that the sums are negligible unless 
\begin{equation}\label{rangemn}
m'\ll \frac{a^2M\ell_a^2D_{2,a}}{\ell^2 Z^{2-\varepsilon}}=\frac{a'^2MD_{2,a}}{Z^{2-\varepsilon}}\qquad\text{and}\qquad n'\ll\frac{b^2N\ell_b^2D_{2,b}}{\ell^2 Z^{2-\varepsilon}}=\frac{b'^2ND_{2,b}}{Z^{2-\varepsilon}}.
\end{equation}
For $m',n'$ in these ranges we bound $G_{a,b}(m',n')$ trivially using $Y_0(x)\ll x^{-1/2}$ and get
\begin{align}\label{boundGab}
G_{a,b}(m',n')&\ll \frac{(D_{2,a}D_{2,b})^{1/4}}{(\ell_a\ell_b)^{1/2}(m'n'MN)^{1/4}}\nonumber\\
&\qquad\qquad\int\int\bigg|\Delta_\ell(ax\mp by-qr)\varphi(ax\mp by-qr)\omega_1\Big(\frac{x}{M}\Big)\omega_2\Big(\frac{y}{N}\Big)\bigg| dxdy\nonumber\\
&=\frac{(D_{2,a}D_{2,b})^{1/4}}{ab(\ell_a\ell_b)^{1/2}(m'n'MN)^{1/4}}\int\int\bigg|\Delta_\ell(u)\varphi(u)\omega_1\Big(\frac{x}{aM}\Big)\omega_2\Big(\frac{\mp(u+qr-x)}{bN}\Big)\bigg| dxdu\nonumber\\
&\ll \frac{(D_{2,a}D_{2,b})^{1/4}}{ab(\ell_a\ell_b)^{1/2}(m'n'MN)^{1/4}}\min\{aM,bN\}\int\big|\Delta_\ell(u)\big| du\nonumber\\
&\ll_\eps Z^\eps \frac{(D_{2,a}D_{2,b})^{1/4}(MN)^{3/4}}{(\ell_a\ell_b)^{1/2}(m'n')^{1/4}(aM+bN)},
\end{align}
by \eqref{bdDelta}. Hence summing over $m',n'$ in the range \eqref{rangemn} we obtain
\[
\sum_{m',n'}\tau(m')\tau(n')\big|G_{a,b}(m',n')\big|\ll_\eps Z^{-3+\eps}\frac{D_{2,a}D_{2,b}(a'b'MN)^{3/2}}{(\ell_a\ell_b)^{1/2}(aM+bN)}.
\]
Furthermore, the sum over $k$ gives a Kloosterman sum, $S_{\psi_{D_{1,a}}\psi_{D_{1,b}}}(qr,*;\ell)$, for which we apply the Weil bound [\textbf{\ref{BC}}; Lemma 3]
\[
S_{\psi_{D_{1,a}}\psi_{D_{1,b}}}(qr,*;\ell)\ll_\varepsilon (qr,\ell)^{1/2}\ell^{1/2+\varepsilon}.
\] Hence 
\begin{align}\label{boundE}
\mathcal{E}_{a,b}\ll_\eps (qMN)^\eps Z^{-5/2}DR\frac{(abMN)^{3/2}}{aM+bN}\ll_\eps (qMN)^\eps q^{-1}Z^{-5/2}D(abMN)^{3/2}.
\end{align}

We now return to the principal term $\mathcal{M}_{a,b}$ in \eqref{afterVoronoi}. This corresponds to the product of the two constants terms after the applications of Theorem \ref{Voronoithm}, and hence
\begin{align}\label{MVoronoi}
\mathcal{M}_{a,b}&=L(1,\psi)^2\sum_{0<|r|\leq R}\sum_{\ell<2Z}\, \sideset{}{^*}\sum_{k(\text{mod}\ \ell)}\rho(ka',\ell_a)\rho(\mp kb',\ell_b)e\Big(\frac{-qrk}{\ell}\Big)\nonumber\\
&\qquad\qquad\int_{0}^{\infty}\int_{0}^{\infty}\Delta_\ell(ax\mp by-qr)\varphi(ax\mp by-qr)\omega_1\Big(\frac{x}{M}\Big)\omega_2\Big(\frac{y}{N}\Big)dxdy,
\end{align}
where, for $(k,\ell)=1$,
\[
\rho(k,\ell)=\frac{1}{\ell}\big(\psi(\ell)+\mathbf{1}_{D|\ell}\tau(\psi)\psi(k)\big).
\]

We first write the sum over $k$ in terms of Ramanujan sums. We have
\begin{align*}
&\rho(ka',\ell_a)\rho(\mp kb',\ell_b)=\frac{1}{\ell_a\ell_b}\big(\psi(\ell_a)+\mathbf{1}_{D|\ell_a}\tau(\psi)\psi(ka')\big)\big(\psi(\ell_b)+\mathbf{1}_{D|\ell_b}\tau(\psi)\psi(kb')\big)\\
&\qquad\qquad=\frac{1}{\ell_a\ell_b}\Big(\psi(\ell_a\ell_b)+\mathbf{1}_{D|(\ell_a,\ell_b)}D\psi(k^2a'b')+\tau(\psi)\psi(k)\big( \mathbf{1}_{D|\ell_a}\psi(a'\ell_b)+\mathbf{1}_{D|\ell_b}\psi(b'\ell_a)\big)\Big)\\
&\qquad\qquad=\frac{1}{\ell_a\ell_b}\Big(\psi(\ell_a\ell_b)+\mathbf{1}_{D|(\ell_a,\ell_b)}D\psi(a'b')+\tau(\psi)\psi(k)\big( \mathbf{1}_{D|(b,\ell)}\psi(a'\ell_b)+\mathbf{1}_{D|(a,\ell)}\psi(kb'\ell_a)\big)\Big).
\end{align*}
As $D\nmid a$ and $D\nmid b$, the last two terms vanish, and hence
\begin{align}\label{sumk}
\sideset{}{^*}\sum_{k(\text{mod}\ \ell)}\rho(ka',\ell_a)\rho(\mp kb',\ell_b)e\Big(\frac{-qkr}{\ell}\Big)=\frac{\big(\psi(\ell_a\ell_b)+\mathbf{1}_{D|(\ell_a,\ell_b)}D\psi(a'b')\big)S(qr,0;\ell)}{\ell_a\ell_b}.
\end{align}

Next we consider the double integrals. By a change of variables we have
\begin{align*}
&\int_{0}^{\infty}\int_{0}^{\infty}\Delta_\ell(ax\mp by-qr)\varphi(ax\mp by-qr)\omega_1\Big(\frac{x}{M}\Big)\omega_2\Big(\frac{y}{N}\Big)dxdy\\
&\qquad\qquad=\frac{1}{ab}\int\int\Delta_\ell(u)\varphi(u)\omega_1\Big(\frac{x}{aM}\Big)\omega_2\Big(\frac{\mp(u+qr-x)}{bN}\Big) dxdu.
\end{align*}
As in \eqref{boundGab}, this is bounded by
\[
\ll_\eps Z^\eps\frac{\min\{aM,bN\}}{ab}\ll_\eps Z^\eps\frac{MN}{aM+bN},
\]
and if $\ell\ll Z^{1-\eps}$, then, by [\textbf{\ref{DFI}}; (18)], this is equal to
\[
\frac{1}{ab}\int\omega_1\Big(\frac{x}{aM}\Big)\omega_2\Big(\frac{\mp(qr-x)}{bN}\Big) dx+O_C\big(Z^{-C}\big)
\]
for any fixed $C>0$. So in view of \eqref{sumk} and the Weil bound, we can first restrict the sum over $\ell$ in \eqref{MVoronoi} to $\ell\ll Z^{1-\eps}$, and then extend it to all $\ell$ at the cost of an error term of size
\begin{align*}
O_\eps\big(L(1,\psi)^2(qMN)^\eps q^{-1}Z^{-1}MN\big).
\end{align*}
Hence
\begin{align}\label{Mab}
\mathcal{M}_{a,b}&=\frac{L(1,\psi)^2}{ab}\sum_{0<|r|\leq R}\mathfrak{S}_{a,b}(qr)\int\omega_1\Big(\frac{x}{aM}\Big)\omega_2\Big(\frac{\mp(qr-x)}{bN}\Big) dx\nonumber\\
&\qquad\qquad+O_\eps\big(L(1,\psi)^2(qMN)^\eps q^{-1}Z^{-1}MN\big),
\end{align}
where $\mathfrak{S}_{a,b}$ is defined in \eqref{mathfrakS}.

Using the well known formula for the Ramanujan sum
\[
S(qr,0;\ell)=\sum_{\substack{\ell=uv\\v|qr}}\mu(u)v
\]
and the fact that $q$ is prime, we can replace $\mathfrak{S}_{a,b}(qr)$ in \eqref{Mab} by $\mathfrak{S}_{a,b}(r)$ with a negligible error term. Also, the condition $|r|\leq R$ can be removed due to the supports of $\omega_1$ and $\omega_2$. Finally, combining \eqref{boundE} and \eqref{Mab}, and choosing $Z=\min\{\sqrt{aM},\sqrt{bN}\}\asymp\sqrt{abMN}(aM+bN)^{-1/2}$ completes the proof.
\end{proof}

\section{The off-diagonal $I^{OD}$}\label{offdiagonal1}

We have
\begin{align}\label{850}
I^{OD}&=\frac{\phi(q)}{2}\sum_{\substack{am\equiv \pm bn(\text{mod}\ q)\\am\ne bn \\a,b\leq X\\D\nmid a,\,D\nmid b\\am,bn>1}}\frac{\rho(a)\rho(b)(1\star\psi)(m)(1\star\psi)(n)}{\sqrt{abmn}}V\Big(\frac{m}{Q}\Big)V\Big(\frac{n}{Q}\Big).
\end{align}
We first show that the conditions $am,bn>1$ may be removed with a negligible error term. Consider the contribution from the terms $bn=1$, which is
\begin{align*}
\frac{\phi(q)}{2}\sum_{\substack{am\equiv \pm 1(\text{mod}\ q)\\am\ne1\\a\leq X\\D\nmid a}}\frac{\rho(a)(1\star\psi)(m)}{\sqrt{am}}V\Big(\frac{m}{Q}\Big)\ll_\eps q^{1/2+\eps}\sum_{\substack{a\leq X\\D\nmid a}}\sum_{\substack{r\geq1\\r\equiv \mp\overline{q}(\text{mod }a)}}\frac{1}{\sqrt{r}}\,\Big|V\Big(\frac{qr\pm1}{aQ}\Big)\Big|.
\end{align*}
This can easily bounded by
\[
\ll_\eps q^{1/2+\eps}\sum_{\substack{a\leq X\\D\nmid a}}\sum_{\substack{r\leq (aQ/q)^{1+\eps}\\r\equiv \mp\overline{q}(\text{mod }a)}}\frac{1}{\sqrt{r}}\ll_\eps q^{1/2+\eps}D^{1/4}X^{1/2},
\]
as desired.

We now apply a dyadic partition of unity to the sums over $a,b,m$ and $n$ in \eqref{850}. Let $\omega$ be a smooth non-negative function supported in $[1, 2]$ such that
$$
\sum_{M} \omega \Big ( \frac{x}{M} \Big ) = 1,
$$
where $M$ runs over a sequence of real numbers with $\#\{M: X^{-1}\leq M\leq X\}\ll \log X$. With this partition of unity, we write
\begin{align*}
I^{OD}=\sum_{A,B,M,N} I^{OD}(A,B,M,N)+O_\eps \big(q^{1/2+\eps}D^{1/4}X^{1/2}\big),
\end{align*}
where 
\begin{align*}
I^{OD}(A,B,M,N)&=\frac{\phi(q)}{2}\sum_{h\geq1}\sum_{\substack{am\equiv \pm bn(\text{mod}\ q)\\am\ne bn\\D\nmid a,\,D\nmid b\\(a,b)=1}}\frac{\rho(ha)\rho(hb)(1\star\psi)(m)(1\star\psi)(n)}{h\sqrt{abmn}}\\
&\qquad\qquad\omega\Big ( \frac{ha}{A} \Big )\omega \Big ( \frac{hb}{B} \Big )\omega\Big ( \frac{m}{M} \Big )\omega \Big ( \frac{n}{N} \Big )V\Big(\frac{m}{Q}\Big)V\Big(\frac{n}{Q}\Big).
\end{align*}
Due to the rapid decay of the function $V$, we may assume that $M,N\ll Q^{1+\varepsilon}$.

We now choose $\delta_0=1/60$ and consider the following two cases.

\subsection{The case $ABMN$ small}\label{unbalanced}

If $ABMN\ll q^{2-2\delta_0}$, then we write $am\equiv \pm bn(\text{mod}\ q)$ and $am\ne bn$ as $am\mp bn=q r$ with $0< |r|\leq R=4(AM+BN)/q.$ Then
\[
I^{OD}(A,B,M,N)\ll_\eps \frac{q^{1+\eps}}{\sqrt{ABMN}}\sum_{0<|r|\leq R}\sum_{h\geq1}\sum_{am\mp bn=qr}\frac1h\bigg|\omega\Big ( \frac{ha}{A} \Big )\omega \Big ( \frac{hb}{B} \Big )\omega\Big ( \frac{m}{M} \Big )\omega \Big ( \frac{n}{N} \Big )\bigg|.
\]
We can take the innermost sum over all $a$ and $m$, and the sums over $b$ and $n$ being over all divisors of $(am-qr)$. So, by symmetry,
\begin{equation*}
I^{OD}(A,B,M,N)\ll_\eps \frac{q^{1+\eps} R}{\sqrt{ABMN}}\min\{AM,BN\}\ll_\eps q^{\eps}\sqrt{ABMN}\ll_\eps q^{1-\delta_0+\eps}.
\end{equation*}

\subsection{The case $ABMN$ large}\label{balanced}

We consider the remaining case $ABMN\gg q^{2-2\delta_0}$. Writing
\[
\omega_1(x)=\omega(x)V\Big(\frac{xM}{Q}\Big)\qquad\text{and}\qquad \omega_2(x)=\omega(x)V\Big(\frac{xN}{Q}\Big).
\]
Applying Proposition \ref{propoff} then yields
\[
I^{OD}(A,B,M,N)=\mathcal{M}(A,B,M,N)+\mathcal{E}(A,B,M,N)
\]
with
\begin{align*}
\mathcal{E}(A,B,M,N)&\ll_\eps q^{\eps}D(AB)^{3/4}(MN)^{-1/4}(AM+BN)^{5/4}\\
&\ll q^{\eps}DX^{11/4}\Big(\frac{M}{N^{1/4}}+\frac{N}{M^{1/4}}\Big)\ll_\eps q^{1+\eps}D^{3/2}X^{11/4}\big(M^{-1/4}+N^{-1/4}\big),
\end{align*}
as $M,N\ll L^{1+\eps}\asymp q^{1+\eps}\sqrt{D}$. Since $ABMN\gg q^{2-2\delta_0}$, it follows that
\[
M,N \gg q^{1-2\delta_0-\eps}D^{-1/2}X^{-2}.
\] 
Thus
\[
\mathcal{E}(A,B,M,N)\ll_\eps  q^{3/4+\delta_0/2+\eps}D^{13/8}X^{13/4}\ll q^{1-\delta_0}.
\]

For $\mathcal{M}(A,B,M,N)$, by Proposition \ref{propoff} we have
\[
\mathcal{M}(A,B,M,N)=\mathcal{M}^+(A,B,M,N)+\mathcal{M}^-(A,B,M,N),
\]
where
\begin{align*}
\mathcal{M}^+(A,B,M,N)&=\frac{\phi(q)}{2}L(1,\psi)^2\sum_{\substack{h\geq1\\(a,b)=1}}\sum_{r\ne 0}\frac{\rho(ha)\rho(hb)\mathfrak{S}_{a,b}(r)}{h(ab)^{3/2}}\omega\Big ( \frac{ha}{A} \Big )\omega \Big ( \frac{hb}{B} \Big )\\
&\qquad\qquad \int (x/a)^{-1/2}(x/b-qr/b)^{-1/2}\,\omega_1\Big(\frac{x}{aM}\Big)\omega_2\Big(\frac{x-qr}{bN}\Big) dx
\end{align*}
and
\begin{align*}
\mathcal{M}^-(A,B,M,N)&=\frac{\phi(q)}{2}L(1,\psi)^2\sum_{\substack{h\geq1\\(a,b)=1}}\sum_{r\ne 0}\frac{\rho(ha)\rho(hb)\mathfrak{S}_{a,b}(r)}{h(ab)^{3/2}}\omega\Big ( \frac{ha}{A} \Big )\omega \Big ( \frac{hb}{B} \Big )\\
&\qquad\qquad \int (x/a)^{-1/2}(qr/b-x/b)^{-1/2}\,\omega_1\Big(\frac{x}{aM}\Big)\omega_2\Big(\frac{qr-x}{bN}\Big) dx.
\end{align*}

Writing $\omega_1$ and $\omega_2$ in terms of their Mellin transforms we get
\begin{align*}
&\mathcal{M}^+(A,B,M,N)=\frac{\phi(q)}{2}\frac{L(1,\psi)^2}{(2\pi i)^2}\int_{(c_2)}\int_{(c_1)}\widetilde{\omega_1}(u)\widetilde{\omega_2}(v)M^{u}N^{v}\\
&\qquad\qquad\sum_{\substack{h\geq1\\(a,b)=1}} \sum_{r\ne 0}\frac{\rho(ha)\rho(hb)\mathfrak{S}_{a,b}(r)}{ha^{1+v}b^{1+u}} \omega\Big ( \frac{ha}{A} \Big )\omega \Big ( \frac{hb}{B} \Big )\int x^{-(1/2+u)}(x-qr/ab)^{-(1/2+v)}dxdudv.
\end{align*}
Note that if $r>0$ then the integral over $x$ is restricted to $x>qr/ab$, and if $r<0$ then it is restricted to $x>0$. For absolute convergence, we shall require the conditions
\begin{equation}\label{contour1}\begin{cases}
\text{Re}(u+v)>0,\ \text{Re}(v)<1/2 & \quad\text{if }r>0,\\
\text{Re}(u+v)>0,\ \text{Re}(u)<1/2 & \quad\text{if }r<0.
\end{cases}\end{equation} Under these assumptions, the $x$-integral is equal to (see, for instance, 17.43.21 and 17.43.22 of [\textbf{\ref{GR}}])
\[
\Big(\frac {q|r|}{ab}\Big)^{-(u+v)}\times\begin{cases}
\frac{\Gamma(u+v)\Gamma(1/2-v)}{\Gamma(1/2+u)} & \quad\text{if }r>0,\\
\frac{\Gamma(u+v)\Gamma(1/2-u)}{\Gamma(1/2+v)} & \quad\text{if } r<0.
\end{cases}
\]
 Hence
\begin{align*}
\mathcal{M}^+(A,B,M,N)&=\frac{\phi(q)}{2}\frac{L(1,\psi)^2}{(2\pi i)^2}\sum_{\substack{h\geq1\\(a,b)=1}}\sum_{r\geq1}\frac{\rho(ha)\rho(hb)\mathfrak{S}_{a,b}(r)}{hab} \omega\Big ( \frac{ha}{A} \Big )\omega \Big ( \frac{hb}{B} \Big )\\
&\qquad\qquad \int_{(c_2)}\int_{(c_1)} \widetilde{\omega_1}(u)\widetilde{\omega_2}(v)\Big(\frac {q}{aM}\Big)^{-u}\Big(\frac {q}{bN}\Big)^{-v}H^+(u,v)r^{-(u+v)}dudv,
\end{align*}
where
\[
H^+(u,v)=\Gamma(u+v)\bigg(\frac{\Gamma(1/2-v)}{\Gamma(1/2+u)}+\frac{\Gamma(1/2-u)}{\Gamma(1/2+v)}\bigg).
\]

A similar formula holds for $\mathcal{M}_1^-(A,B,M,N)$ with $H^+$ being replaced by $H^-$, where
\[
H^-(u,v)=\frac{\Gamma(1/2-u)\Gamma(1/2-v)}{\Gamma(1-u-v)},
\]
and the imposed conditions for the absolute convergence are
\begin{equation}\label{contour2}
\text{Re}(u),\ \text{Re}(v)<1/2.
\end{equation}
So
\begin{align*}
&\mathcal{M}^+(A,B,M,N)+\mathcal{M}^-(A,B,M,N)\\
&\qquad\qquad =\frac{\phi(q)}{2}L(1,\psi)^2\sum_{\substack{h\geq1\\(a,b)=1}}\frac{\rho(ha)\rho(hb)}{hab} \omega\Big ( \frac{ha}{A} \Big )\omega \Big ( \frac{hb}{B} \Big )I_{a,b}(M,N),
\end{align*}
where
\begin{align}\label{Nintegral}
I_{a,b}(M,N)= \frac{1}{(2\pi i)^2}\sum_{r\geq1}\mathfrak{S}_{a,b}(r)\int_{(c_2)}\int_{(c_1)} \widetilde{\omega_1}(u)\widetilde{\omega_2}(v)\Big(\frac {q}{aM}\Big)^{-u}\Big(\frac {q}{bN}\Big)^{-v}H(u,v)r^{-(u+v)}dudv
\end{align}
with
\begin{align*}
H(u,v)=\Gamma(u+v)\bigg(\frac{\Gamma(1/2-v)}{\Gamma(1/2+u)}+\frac{\Gamma(1/2-u)}{\Gamma(1/2+v)}\bigg)+\frac{\Gamma(1/2-u)\Gamma(1/2-v)}{\Gamma(1-u-v)}.
\end{align*}

\subsubsection{The $r$-sum}

Next we want to replace the sum over $r$ in \eqref{Nintegral} by the Riemann zeta-function. We note from [\textbf{\ref{T}}; (1.5.5)] that
\begin{align*}
\sum_{r\geq1}\frac{S(r,0;\ell)}{r^s}=\zeta(s)\sum_{d|\ell}\mu\Big(\frac\ell d\Big)d^{1-s}=\zeta(s)l^{1-s}\prod_{p|\ell}(1-p^{s-1})
\end{align*}
for $\text{Re}(s)>1$. So we get
\begin{align*}
\sum_{\ell\geq1}\frac{\psi(\ell_a\ell_b)}{\ell_a\ell_b}\sum_{r\geq1}\frac{S(r,0;\ell)}{r^s}&=\zeta(s)\bigg(\sum_{\substack{a_1|a\\b_1|b}}\frac{\psi(a_1b_1)}{a_1^sb_1^s}\prod_{p|a_1b_1}(1-p^{s-1})\bigg)\sum_{(\ell,abD)=1}\frac{1}{\ell^{1+s}}\prod_{p|\ell}(1-p^{s-1})\\
&=\zeta(s)\bigg(\sum_{\substack{a_1|a\\b_1|b}}\frac{\psi(a_1b_1)}{a_1^sb_1^s}\prod_{p|a_1b_1}(1-p^{s-1})\bigg)\prod_{p\nmid abD}\bigg(1-\frac{1}{p^{1+s}}\bigg)^{-1}\bigg(1-\frac{1}{p^2}\bigg).
\end{align*}
Similarly,
\begin{align*}
&\sum_{\ell\geq1}\frac{\mathbf{1}_{D|(\ell_a,\ell_b)}D\psi(a'b')}{\ell_a\ell_b}\sum_{r\geq1}\frac{S(r,0;\ell)}{r^s}\\
&\qquad\qquad=\mathbf{1}_{(D,ab)=1}\zeta(s)\bigg(\sum_{\substack{a_1|a\\b_1|b}}\frac{\psi(ab/a_1b_1)}{a_1^sb_1^sD^s}\prod_{p|a_1b_1D}(1-p^{s-1})\bigg)\prod_{p\nmid abD}\bigg(1-\frac{1}{p^{1+s}}\bigg)^{-1}\bigg(1-\frac{1}{p^2}\bigg).
\end{align*}
Hence for $\text{Re}(s)>1$,
\begin{align*}
\sum_{r\geq1}\frac{\mathfrak{S}_{a,b}(r)}{r^s}&=\sum_{\ell\geq1}\frac{\psi(\ell_a\ell_b)+\mathbf{1}_{D|(\ell_a,\ell_b)}D\psi(a'b')}{\ell_a\ell_b}\sum_{r\geq1}\frac{S(r,0;\ell)}{r^s}=G_{a,b}(s)\zeta(s)\zeta(1+s),
\end{align*}
where, in particular, $G_{a,b}(1)=1/\zeta(2)$, and 
\begin{equation}\label{estimateG}
G_{a,b}(s)\ll_{\eps} (\log abD)^{\eps}\tau(a)\tau(b)
\end{equation}
uniformly for $|\text{Re}(s)|\ll 1/\log q$.

We move the $u$-contour in \eqref{Nintegral} to the right to $\text{Re}(u)=1$, crossing a simple pole at $u=1/2$ from the gamma factors. In doing so we obtain 
\[
I_{a,b}(M,N)=I_{a,b}'(M,N)-\mathcal{R}_{a,b}(M,N),
\]
where $I_{a,b}'(M,N)$ is the new integral and
\begin{align*}
\mathcal{R}_{a,b}(M,N)&=-\frac{2\widetilde{\omega_1}(1/2)}{2\pi i}\Big(\frac {q}{aM}\Big)^{-1/2}\sum_{r\geq1}\mathfrak{S}_{a,b}(r)\int_{(c_2)}\widetilde{\omega_2}(v)\Big(\frac {q}{bN}\Big)^{-v}r^{-(1/2+v)}dv.
\end{align*}
As moving the $v$-contour in the above expression to $\text{Re}(v)=1$ encounters no pole, we have
\begin{align}\label{R1}
&\mathcal{R}_{a,b}(M,N)=-\frac{2\widetilde{\omega_1}(1/2)}{2\pi i}\Big(\frac {q}{aM}\Big)^{-1/2}\int_{(1)}\widetilde{\omega_2}(v)\Big(\frac {q}{bN}\Big)^{-v}G_{a,b}(\tfrac12+v)\zeta(\tfrac12+v)\zeta(\tfrac32+v)dv.
\end{align}
With $I_{a,b}'(M,N)$, the $r$-sum can be written as $G_{a,b}(u+v)\zeta(u+v)\zeta(1+u+v)$. We then shift the $u$-contour back to $\text{Re}(u)=c_1$, encountering a simple pole at $u=1/2$ again. Note that $H(1-v,v)=0$, which cancels out the pole at $u=1-v$ of the zeta-function. We write
\[
I_{a,b}'(M,N)=I_{a,b}''(M,N)+\mathcal{R}_{a,b}'(M,N)
\]
accordingly, where $\mathcal{R}_{a,b}'(M,N)$ is the residue at $u=1/2$. It follows that
\[
I_{a,b}(M,N)=I_{a,b}''(M,N)-\big(\mathcal{R}_{a,b}(M,N)-\mathcal{R}_{a,b}'(M,N)\big).
\]

Note that $\mathcal{R}_{a,b}'(M,N)$ is the same as $\mathcal{R}_{a,b}(M,N)$ in \eqref{R1} but with the $v$-contour being along $\text{Re}(v)=c_2$. So the difference $\mathcal{R}_{a,b}(M,N)-\mathcal{R}_{a,b}'(M,N)$ is the residue at $v=1/2$ of \eqref{R1}, which is
\begin{align*}
-\frac{2\widetilde{\omega_1}(1/2)\widetilde{\omega_2}(1/2)\sqrt{abMN}}{q},
\end{align*}
as $G_{a,b}(1)=1/\zeta(2)$. Thus,
\begin{align}\label{IabMN}
&I_{a,b}(M,N)= \frac{1}{(2\pi i)^2}\int_{(c_2)}\int_{(c_1)} \widetilde{\omega_1}(u)\widetilde{\omega_2}(v)\Big(\frac {q}{aM}\Big)^{-u}\Big(\frac {q}{bN}\Big)^{-v}\nonumber\\
&\qquad\qquad H(u,v)G_{a,b}(u+v)\zeta(u+v)\zeta(1+u+v)dudv -\frac{2\widetilde{\omega_1}(1/2)\widetilde{\omega_2}(1/2)\sqrt{abMN}}{q}.
\end{align}

\subsection{Assembling the partition of unity}

Recall from Subsection \ref{unbalanced} that 
\begin{equation*}
I^{OD}(A,B,M,N)\ll_\eps q^{1-\delta_0+\eps}
\end{equation*}
if $ABMN\ll q^{2-2\delta_0}$, and from Subsection \ref{balanced} that
\begin{align*}
I^{OD}(A,B,M,N)&=\frac{\phi(q)}{2}L(1,\psi)^2\sum_{\substack{h\geq1\\(a,b)=1}}\frac{\rho(ha)\rho(hb)}{hab} \omega\Big ( \frac{ha}{A} \Big )\omega \Big ( \frac{hb}{B} \Big )I_{a,b}(M,N)+O\big(q^{1-\delta_0}\big)
\end{align*}
if $ABMN\gg q^{2-2\delta_0}$. Summing up we obtain
\begin{align}\label{1stbdIOD}
I^{OD}=\frac{\phi(q)}{2}L(1,\psi)^2\sum_{\substack{A,B,M,N\\ABMN\gg q^{2-2\delta_0}}}\sum_{\substack{h\geq1\\(a,b)=1}}\frac{\rho(ha)\rho(hb)}{hab} \omega\Big ( \frac{ha}{A} \Big )\omega \Big ( \frac{hb}{B} \Big )I_{a,b}(M,N)+O_\eps\big(q^{1-\delta_0+\eps}\big).
\end{align} 

The following result shall allow us to add all the missing tuples $(A, B,M,N)$, and thus to remove the partition of unity.

\begin{lemma}
With $I_{a,b}(M,N)$ defined as in \eqref{IabMN} we have
\[
I_{a,b}(M,N)\ll_\eps q^{-1+\eps}\sqrt{abMN}.
\]
\end{lemma}
\begin{proof}
The second term of $I_{a,b}(M,N)$ clearly satisfies the above bound. For the first term, recall from \eqref{contour1} and \eqref{contour2} that $H(u,v)$ has rapid decay as any of the variables gets large in the imaginary directions. So we have, trivially, 
\[
I_{a,b}(M,N)\ll_\eps \Big(\frac {q}{aM}\Big)^{-c_1}\Big(\frac {q}{bN}\Big)^{-c_2}+q^{-1+\eps}\sqrt{abMN},
\]
provided that
\[
c_1+c_2>0,\qquad c_1,\, c_2<\frac12.
\]
The lemma follows by choosing $c_1=c_2=1/2-\eps$.
\end{proof}

The above lemma and \eqref{1stbdIOD} imply that
\begin{align*}
I^{OD}&=\frac{\phi(q)}{2}L(1,\psi)^2\sum_{h\leq X}\frac 1h\sum_{\substack{a,b\leq X/h\\D\nmid a,\,D\nmid b\\(a,b)=1}}\frac{\rho(ha)\rho(hb)}{ab}\sum_{M,N}I_{a,b}(M,N)+O_\eps\big(q^{1-\delta_0+\eps}\big).
\end{align*}
Note that
\[
M^u\widetilde{\omega_1}(u)=\int_{0}^{\infty}\omega_1\Big(\frac xM\Big)V\Big(\frac{x}{Q}\Big)x^{u-1}dx.
\]
So we have
\begin{align}\label{700}
I^{OD}&=\mathcal{M}_{1}^{OD}+\mathcal{M}_{2}^{OD}+O_\eps\big(q^{1-\delta_0+\eps}\big),
\end{align}
where
\begin{align*}
\mathcal{M}_{1}^{OD}= \frac{\phi(q)}{2}L(1,\psi)^2\sum_{h\leq X}\frac 1h\sum_{\substack{a,b\leq X/h\\D\nmid a,\,D\nmid b\\(a,b)=1}}\frac{\rho(ha)\rho(hb)}{ab}I_{a,b}
\end{align*}
with
\begin{equation*}
I_{a,b}=\frac{1}{(2\pi i)^2}\int_{(c_2)}\int_{(c_1)}\widetilde{V}(u)\widetilde{V}(v)\big(\sqrt{D}a\big)^{u}\big(\sqrt{D}b\big)^{v}H(u,v) G_{a,b}(u+v)\zeta(u+v)\zeta(1+u+v)dudv,
\end{equation*}
and
\begin{align}\label{701}
\mathcal{M}_{2}^{OD}&=\phi(q) L(1,\psi)^2\widetilde{V}(\tfrac12)^2\sqrt{D}\sum_{h\leq X}\frac 1h\sum_{\substack{a,b\leq X/h\\D\nmid a,\,D\nmid b\\(a,b)=1}}\frac{\rho(ha)\rho(hb)}{\sqrt{ab}}\nonumber\\
&= L(1,\psi)^2\widetilde{V}(\tfrac12)^2Q\sum_{\substack{a,b\leq X\\D\nmid a,\,D\nmid b}}\frac{\rho(a)\rho(b)}{\sqrt{ab}}+O_\eps\big(L(1,\psi)^2q^{\eps}\sqrt{D}X\big).
\end{align}
Note that $\mathcal{M}_{2}^{OD}$ cancels out the third term in \eqref{mainformulaI} up to an error of size $O_\eps(L(1,\psi)^2q^{\eps}\sqrt{D}X)$.

It remains to evaluate $\mathcal{M}_{1}^{OD}$. From Lemma 8.2 of [\textbf{\ref{Y}}], we have
\[
H(u,v)=\pi^{1/2}\frac{\Gamma(\frac{u+v}{2})\Gamma(\frac{1/2-u}{2})\Gamma(\frac{1/2-v}{2})}{\Gamma(\frac{1-u-v}{2})\Gamma(\frac{1/2+u}{2})\Gamma(\frac{1/2+v}{2})}.
\] 
By the functional equation of the Riemann zeta-function,
\[
\pi^{1/2}\frac{\Gamma(\frac{u+v}{2})}{\Gamma(\frac{1-u-v}{2})}\zeta(u+v)=\pi^{u+v}\zeta(1-u-v),
\]
it follows that
\begin{align*}
I_{a,b}&=\frac{1}{(2\pi i)^2}\int_{(c_2)}\int_{(c_1)}\widetilde{V}(u)\widetilde{V}(v)\frac{\Gamma(\frac{1/2-u}{2})\Gamma(\frac{1/2-v}{2})}{\Gamma(\frac{1/2+u}{2})\Gamma(\frac{1/2+v}{2})}\\
&\qquad\qquad\qquad\qquad\qquad\big(\pi\sqrt{D}a\big)^{u}\big(\pi\sqrt{D}b\big)^{v} G_{a,b}(u+v)\zeta(1-u-v)\zeta(1+u+v)dudv.
\end{align*}
We choose $c_1=c_2=1/\log q$ and estimate the double integrals trivially using \eqref{estimateG}. In doing so we get
\[
I_{a,b}\ll_\eps \tau(a)\tau(b)(\log q)^{4+\eps}.
\]
Thus
\[
\mathcal{M}_{1}^{OD}\ll_\eps L(1,\psi)^2q(\log q)^{16+\eps},
\]
and combining with \eqref{700} and \eqref{701} leads to
\begin{align}\label{boundIOD}
I^{OD}&= L(1,\psi)^2\widetilde{V}(\tfrac12)^2Q\sum_{\substack{a,b\leq X\\D\nmid a,\,D\nmid b}}\frac{\rho(a)\rho(b)}{\sqrt{ab}}+O_\eps \big(L(1,\psi)^2q(\log q)^{16+\eps}\big).
\end{align}

Proposition \ref{mainprop} follows from \eqref{mainformulaI}, \eqref{boundID} and \eqref{boundIOD}.

\appendix
\section{Voronoi summation formula}

Assume $D > 1$ is a squarefree fundamental discriminant. We let $\psi$ denote a real primitive character modulo $D$, so that $\psi$ is an even Dirichlet character. We shall prove the following summation formula.

\begin{theorem}\label{Voronoithm}
Let $D$ and $\psi$ be as above and let $(a,c)=1$. Let $\psi_1$ be a real primitive character modulo $(c,D)$ and let $\psi_2$ be a real primitive character modulo $D_c=D/(c,D)$ such that $\psi = \psi_1 \psi_2$. For any smooth function $g$ compactly supported in $\mathbb{R}_{>0}$, we have
\begin{align}\label{eqn:summation formula}
\sum_n (1\star\psi)(n) e \Big(\frac{an}{c} \Big) g(n) &= \rho(a,c) L(1,\psi) \int_0^\infty g(x) dx + T(a,c),
\end{align}
where
\begin{align*}
\rho(a,c) &= 
\begin{cases}
\psi(c)/c  &\text{if }D\nmid c, \\
\tau(\psi)\psi(a)/c &\text{if } D| c
\end{cases}
\end{align*}
and
\begin{align*}
T(a,c) &= -2\pi\tau(\psi_2)\frac{\psi_1(-a) \psi_2(c)}{cD_c}\sum_m (\psi_1 \star \psi_2)(m) e \Big(\frac{-\overline{a D_c}m}{c} \Big)\int_0^\infty g(x) Y_0 \Big(\frac{4\pi\sqrt{mx}}{c\sqrt{D_c}} \Big)dx \\
&\qquad\qquad+4\tau(\psi_2)\frac{\psi_1(a) \psi_2(c)}{cD_c} \sum_m (\psi_1 \star \psi_2)(m) e \Big(\frac{\overline{a D_c}m}{c} \Big)\int_0^\infty g(x) K_0  \Big(\frac{4\pi\sqrt{mx}}{c\sqrt{D_c}} \Big)dx.
\end{align*}
Here $\tau(\chi)$ denotes the Gauss sum for the character $\chi$.
\end{theorem}
\begin{proof}
When $D| c$ we may appeal directly to [\textbf{\ref{IK}}; Theorem 4.13], and when $(D,c)=1$ we may apply [\textbf{\ref{IK}}; Theorem 4.14]. Therefore, we may assume that $1<(c,D) < D$.

For technical reasons it is convenient to prove a slightly more general result, from which \eqref{eqn:summation formula} will follow. Given $\nu \in \mathbb{C}$ we define
\begin{align*}
\tau_\nu(n,\psi) := \sum_{m_1m_2 = n} \psi(m_2) \Big(\frac{m_2}{m_1} \Big)^\nu,
\end{align*}
and then
\begin{align*}
S := \sum_n \tau_\nu(n,\psi)e \Big(\frac{an}{c} \Big) g(n).
\end{align*}
Observe that when $\nu = 0$ we recover the left side of $\eqref{eqn:summation formula}$. 

We open the function $\tau_\nu(n,\psi)$ and apply smooth partitions of unity on the resulting sums. This gives
\begin{align*}
S= \sum_{M,N}\sum_{m,n} \psi(n)  \Big(\frac{n}{m} \Big)^\nu e \Big(\frac{amn}{c} \Big)g(mn) \omega\Big( \frac{m}{M} \Big) \omega \Big(\frac{n}{N} \Big) =\sum_{M,N} S_{M,N},
\end{align*}
say. We then split the sum over $m$ into arithmetic progressions modulo $c$, and the sum over $n$ into progressions modulo $c$ and $D$. This yields
\begin{align*}
S_{M,N} &= \sum_{u(\textup{mod }c)} \sum_{v(\textup{mod }c)} e \Big(\frac{auv}{c} \Big)\sum_{r(\textup{mod }D)} \psi(r) \sum_{m \equiv u (\textup{mod }c)}\sum_{\substack{n \equiv v (\textup{mod }c) \\ n \equiv r (\textup{mod }D)}}\Big(\frac{n}{m} \Big)^\nu g(mn) \omega\Big( \frac{m}{M} \Big) \omega \Big(\frac{n}{N} \Big).
\end{align*}
In order for the congruences on $n$ to be compatible, we must have $r \equiv v(\text{mod}\ (c,D))$. We may use the Chinese Remainder Theorem to combine the congruences on $n$ into a single congruence $n \equiv \gamma(\text{mod}\ [c,D])$. Performing Poisson summation in $m$ and $n$, we obtain
\begin{align*}
S_{M,N} &= \frac{1}{c[c,D]} \sum_{h,k \in \mathbb{Z}}\sum_{u(\textup{mod }c)} \sum_{v(\textup{mod }c)}e \Big(\frac{auv}{c} \Big)\sum_{\substack{r(\textup{mod }D) \\ r \equiv v (\textup{mod }(c,D))}} \psi(r) e \Big(\frac{hu}{c} \Big)e \Big(\frac{k\gamma}{[c,D]} \Big) \\
&\qquad\qquad \int_0^\infty \int_0^\infty \Big(\frac{y}{x} \Big)^\nu e \Big(\frac{-hx}{c} \Big)e \Big(\frac{-ky}{[c,D]} \Big) g(xy)\omega\Big( \frac{x}{M} \Big) \omega \Big(\frac{y}{N} \Big) dx dy.
\end{align*}
We then perform the summation in $u$ to obtain
\begin{align*}
S_{M,N} &= \frac{1}{[c,D]}\sum_{h,k \in \mathbb{Z}}\sum_{\substack{r(\textup{mod }D) \\ r \equiv -\overline{a}h (\textup{mod }(c,D))}} \psi(r) e \Big(\frac{k\gamma}{[c,D]} \Big) \\
&\qquad\qquad \int_0^\infty \int_0^\infty \Big(\frac{y}{x} \Big)^\nu e \Big(\frac{-hx}{c} \Big)e \Big(\frac{-ky}{[c,D]} \Big) g(xy)\omega\Big( \frac{x}{M} \Big) \omega \Big(\frac{y}{N} \Big) dx dy,
\end{align*}
where the congruence class $\gamma = \gamma(r,h)$ satisfies $\gamma \equiv -\overline{a}h(\text{mod}\ c), \gamma \equiv r (\text{mod}\ D)$.

There is no contribution to $S_{M,N}$ from $h = 0$. Indeed, we then have $r \equiv 0 (\textup{mod }(c,D))$, but $(c,D) > 1$, and therefore $(r,D) > 1$, so $\psi(r) = 0$. We may therefore assume that $h \neq 0$.

We next examine the contribution from $k=0, h \neq 0$. We sum over $N$ to remove the factors $\omega(y/N)$, and arrive at
\begin{align*}
\sum_M \frac{1}{[c,D]} \mathop{\sum}_{h\neq 0}\sum_{\substack{r(\textup{mod }D) \\ r \equiv -\overline{a}h (\textup{mod }(c,D))}} \psi(r)\int_0^\infty \int_0^\infty \left(\frac{y}{x} \right)^\nu e \Big(\frac{-hx}{c} \Big)g(xy)\omega\Big( \frac{x}{M} \Big) dx dy.
\end{align*}
Since $D \nmid c$ and $\psi$ is a primitive character, we have
\begin{align*}
\sum_{\substack{r(\textup{mod }D) \\ r \equiv -\overline{a}h (\textup{mod }(c,D))}} \psi(r) = 0.
\end{align*}

It therefore remains to handle the contribution from $hk \neq 0$. The main task is to compute a closed formula for the sum
\begin{align*}
T := \sum_{\substack{r(\textup{mod }D) \\ r \equiv -\overline{a}h (\textup{mod }(c,D))}} \psi(r)e \Big(\frac{k\gamma}{[c,D]} \Big) .
\end{align*}
We have $[c,D] = cD_c$, and since $(c,D_c) = 1$ we may use reciprocity to write
\begin{align*}
\frac{k\gamma}{[c,D]} &= - \frac{\overline{a D_c} hk}{c} + \frac{kr \overline{c}}{D_c} \ (\textup{mod }1).
\end{align*}
It follows that
\begin{align*}
T &= e \Big(\frac{-\overline{a D_c}hk}{c} \Big) \sum_{\substack{r(\textup{mod }D) \\ r \equiv -\overline{a}h (\textup{mod }(c,D))}} \psi(r) e \Big(\frac{kr \overline{c}}{D_c} \Big).
\end{align*}
We have $D = (c,D) D_c$, and since $D$ is squarefree we have $((c,D),D_c)=1$. We may then factor $\psi = \psi_1 \psi_2$, where $\psi_1$ is a primitive character modulo $(c,D)$ and $\psi_2$ is a primitive character modulo $D_c$. We may use the Chinese Remainder Theorem to write $r(\text{mod}\ D)$ as $\alpha D_c + \beta (c,D)$, where $\alpha$ ranges over $(\mathbb{Z}/(c,D)\mathbb{Z})^\times$ and $\beta$ ranges over $(\mathbb{Z}/D_c\mathbb{Z})^\times$. We therefore have $\alpha \equiv - \overline{a D_c}h(\text{mod}\ (c,D))$, which fixes $\alpha$. Furthermore, we have
\begin{align*}
\psi(r) = \psi_1(\alpha D_c) \psi_2\big(\beta (c,D)\big),
\end{align*}
so that
\begin{align*}
T &= e \Big(\frac{-\overline{a D_c}hk}{c} \Big) \psi_1(-ah) \psi_2\big((c,D)\big)\sum_{\substack{\beta (\textup{mod }D_c) }} \psi_2(\beta)e \Big(\frac{k \beta (c,D) \overline{c}}{D_c} \Big).
\end{align*}
The sum over $\beta$ is a Gauss sum for $\psi_2$. Since $\psi_2$ is primitive we have
\begin{align*}
\sum_{\substack{\beta (\textup{mod }D_c) }} \psi_2(\beta)e \Big(\frac{k \beta (c,D) \overline{c}}{D_c} \Big) &= \tau(\psi_2)\psi_2\big(ck(c,D)\big) ,
\end{align*}
and therefore
\begin{align*}
T &=\tau(\psi_2) e \Big(\frac{-\overline{a D_c}hk}{c} \Big) \psi_1(-ah) \psi_2(ck) .
\end{align*}

After reassembling the smooth partitions of unity we find that
\begin{align*}
S &= \tau(\psi_2)\frac{\psi_1(-a)\psi_2(c)}{cD_c} \sum_{\substack{h,k\in\mathbb{Z}\\hk \neq 0}} e \Big(\frac{-\overline{a D_c}hk}{c} \Big)\psi_1(h) \psi_2(k)\\
&\qquad\qquad\int_0^\infty \int_0^\infty  \Big( \frac{y}{x}\Big)^\nu e \Big(\frac{-hx}{c} \Big) e \Big(\frac{-ky}{[c,D]} \Big) g(xy)dx dy.
\end{align*}
Define the integral
\begin{align*}
I(\epsilon_1 A,\epsilon_2 B) := \int_0^\infty \int_0^\infty  \Big( \frac{y}{x}\Big)^\nu e \left(\epsilon_1 A x +\epsilon_2 By \right) g(xy)dx dy,
\end{align*}
where $\epsilon_i \in \{\pm 1\}$ and $A,B > 0$ are real numbers. Changing the variables
\begin{align*}
(x,y) \longrightarrow \bigg(u \sqrt{\frac{v B}{A}}, \frac{1}{u} \sqrt{\frac{vA}{B}} \bigg)
\end{align*}
gives
\begin{align*}
I(\epsilon_1 A, \epsilon_2 B) = \Big(\frac{A}{B} \Big)^\nu \int_0^\infty g(v) \int_0^\infty u^{-(2\nu+ 1)} e \big(\sqrt{vAB} (\epsilon_1 u + \epsilon_2 u^{-1}) \big) du dv.
\end{align*}

The sum on $h$ and $k$ can be re-written as
\begin{align*}
&\sum_{h,k \geq 1} e \Big(\frac{-\overline{a D_c}hk}{c} \Big) \psi_1(h)\psi_2(k) \bigg(I\Big(\frac{-h}{c},\frac{-k}{[c,D]} \Big) + \psi(-1)I\Big(\frac{h}{c},\frac{k}{[c,D]} \Big) \bigg) \\
&\qquad\qquad+ \sum_{h,k \geq 1}e \Big(\frac{-\overline{a D_c}hk}{c} \Big) \psi_1(h)\psi_2(k) \bigg(\psi_2(-1)I\Big(\frac{-h}{c},\frac{k}{[c,D]} \Big) + \psi_1(-1)I\Big(\frac{h}{c},\frac{-k}{[c,D]} \Big) \bigg).
\end{align*}
We finish the proof by writing $I(\cdot,\cdot)$ as an integral transform of $g$ by arguing as in the proof of Theorem 4.13 in [\textbf{\ref{IK}}] and using the integral identities (4.112)--(4.115) in [\textbf{\ref{IK}}]. We obtain \eqref{eqn:summation formula} via analytic continuation on sending $\nu \rightarrow 0$.
\end{proof}

\section*{Acknowledgments}
During part of this work the second author was supported by the National Science Foundation Graduate Research Program under grant number DGE-1144245.

The authors thank the anonymous referee for their comments on a previous version of this manuscript.

\section*{Conflict of interest}

On behalf of all authors, the corresponding author states that there is no conflict of interest.

\section*{Data availability}

Data sharing not applicable to this article as no datasets were generated or analysed during the current study.


\begin{thebibliography}{99}
\bibitem{BC}\label{BC}
S. Bettin, V. Chandee, {\it Trilinear forms with Kloosterman fractions}, Adv. Math. \textbf{328} (2018), 1234--1262.

\bibitem{BFKMM}\label{BFKMM}
V. Blomer, \'E. Fouvry, E. Kowalski, P. Michel, D. Mili\'cevi\'c, \emph{On moments of twisted $L$-functions}, Amer. J. Math. \textbf{139} (2017), 707--768.

\bibitem{B}\label{B}
H. M. Bui, {\it Non-vanishing of Dirichlet L-functions at the central point}, Int. J. Number Theory \textbf{8} (2012), 1855--1881.
\bibitem{BM}\label{BM}
H. M. Bui, M. B. Milinovich, {\it Central values of derivatives of Dirichlet L-functions}, Int. J. Number Theory \textbf{7} (2011), 371--388.
\bibitem{BPRZ}\label{BPRZ}
H. M. Bui, K. Pratt, N. Robles, A. Zaharescu, {\it Breaking the $\frac12$-barrier for the twisted second moment of Dirichlet L-functions}, Adv. Math. \textbf{370} (2020), 107175, 40 pp.
\bibitem{CI}\label{CI}
J. B. Conrey, H. Iwaniec, {\it Critical zeros of lacunary L-functions}, Acta Arith. \textbf{195} (2020), no. 3, 217--268.

\bibitem{Dav}\label{Dav}
H.
 Davenport, \emph{Multiplicative number theory}, 3rd edition, revised and with a preface by H. L. Montgomery, Graduate Texts in Mathematics, vol. 74, Springer-Verlag, New York (2000).

\bibitem{DM19}\label{DM19}
S. Drappeau, J. Maynard, \emph{Sign changes of Kloosterman sums and exceptional characters}, Proc. Amer. Math. Soc. \textbf{147} (2019), 61--75.

\bibitem{DFI}\label{DFI} 
W. Duke, J. B. Friedlander, H. Iwaniec, {\it  A quadratic divisor problem}, Invent. Math. \textbf{115} (1994), 209--217.

\bibitem{FI03}\label{FI03}
J. B. Friedlander, H. Iwaniec, \emph{Exceptional characters and prime numbers in arithmetic progressions}, Int. Math. Res. Not., no. 37 (2003), 2033--2050.

\bibitem{FI17}\label{FI17}
J. B. Friedlander, H. Iwaniec, \emph{The temptation of the exceptional characters}. \emph{Exploring the Riemann zeta function}, Springer, Cham (2017), 67--81.

\bibitem{FI18}\label{FI18}
J. B. Friedlander, H. Iwaniec, \emph{A note on Dirichlet $L$-functions}. Expo. Math. 36 (2018), no. 3-4, 343--350.

\bibitem{GR}\label{GR}
I. S. Gradshteyn, I. M. Ryzhik, {\it Table of integrals, series, and products}, Academic Press, New York (1965).

\bibitem{HB83}\label{HB83}
D. R. Heath-Brown, \emph{Prime twins and Siegel zeros}, Proc. London Math. Soc. \textbf{47} (1983), 193--224.

\bibitem{H}\label{H}
B. Hough, {\it The angle of large values of L-functions}, J. Number Theory \textbf{167} (2016), 353--393.

\bibitem{Iw06}\label{Iw06}
H. Iwaniec, {\it Conversations on the exceptional character}. \emph{Analytic number theory}, 
Lecture Notes in Math., vol. 1891, Springer, Berlin (2006), 97--132.

\bibitem{IK}\label{IK} 
H. Iwaniec, E. Kowalski, {\it Analytic Number Theory}, American Mathematical Society Colloquium Publications, vol. 53, American Mathematical Society, Providence, RI (2004).
\bibitem{IS}\label{IS}
H. Iwaniec, P. Sarnak, \textit{Dirichlet L-functions at the central point}, Number Theory in Progress, vol. 2, de Gruyter, Berlin (1999), 941--952.

\bibitem{KS99}\label{KS99}
N. M. Katz, P. Sarnak, \emph{Zeroes of zeta functions and symmetry}, Bull. Amer. Math. Soc. \textbf{36} (1999), 1--26.

\bibitem{KN}\label{KN}
R. Khan, H. Ngo, {\it Nonvanishing of Dirichlet L-functions}, Algebra Number Theory \textbf{10} (2016), 2081--2091.
\bibitem{L}\label{L}
E. Landau, {\it Bemerkungen zum Heilbronnschen Satz}, Acta Arith. \textbf{1} (1935), 1--18.

\bibitem{MV}\label{MV}
P. Michel, J. VanderKam, {\it Non-vanishing of high derivatives of Dirichlet L-functions at
the central point}, J. Number Theory \textbf{81} (2000), 130--148.

\bibitem{Murty90}\label{Murty90}
M. R. Murty, \emph{On simple zeros of certain L-series}, Number theory (Banff, AB, 1988), de Gruyter, Berlin (1990), 427--439.

\bibitem{Si}\label{Si}
C. L. Siegel, {\it \" Uber die Classenzahl quadratischer Zahlk\" orper}, Acta Arith. \textbf{1} (1935), 83--86.

\bibitem{T}\label{T} 
E. C. Titchmarsh, {\it The theory of the Riemann zeta-function}, $2$nd edition, edited and with a preface by D. R. Heath-Brown, The Clarendon Press, Oxford University Press (1986).

\bibitem{Wiles}\label{Wiles}
Andrew Wiles. The Birch and Swinnerton-Dyer conjecture, \emph{The millennium prize problems}, 31--41, Clay Math. Inst., Cambridge, MA, 2006. 

\bibitem{Y}\label{Y}
M. P. Young, {\it The fourth moment of Dirichlet L-functions}, Ann. of Math. \textbf{173} (2011), 1--50.
\bibitem{Z}\label{Z}
R. Zacharias, {\it Mollification of the fourth moment of Dirichlet L-functions}, Acta Arith. \textbf{191} (2019), no. 3, 201--257.
\end{thebibliography}
\end{document}